\documentclass[reqno,11pt]{amsart}
\usepackage{amsmath,amssymb,amsfonts}

\usepackage{enumerate}

\usepackage{tikz}
\usetikzlibrary{matrix,arrows}
\usepackage{verbatim}
\usepackage[pagebackref,colorlinks,citecolor={red!50!blue},linkcolor={blue!80!black}]{hyperref}

\renewcommand*{\backref}[1]{}
\renewcommand*{\backrefalt}[4]{%
    \scriptsize%
    {
    \ifcase #1 (\textcolor{red}{Uncited.})%
          \or (Cited\ on p.~#2)%
          \else (Cited\ on pp.~#2)%
    \fi%
    }
}

\usepackage{fourier}
\usepackage[utf8]{inputenc}

\title{Relative perfect complexes}

\author[L. Alonso]{Leovigildo Alonso Tarr\'{\i}o}
\address[L. A. T.]{CITMAGA\\
Departamento de Matem\'a\-ticas\\
Universidade de Santiago de Compostela\\
E-15782  Santiago de Compostela, Spain}
\email{leo.alonso@usc.es}

\author[A. Jerem\'{\i}as]{Ana Jerem\'{\i}as L\'opez}
\address[A. J. L.]{CITMAGA\\
Departamento de Matem\'a\-ticas\\
Universidade de Santiago de Compostela\\
E-15782  Santiago de Compostela, Spain}
\email{ana.jeremias@usc.es}

\author[F. Sancho]{Fernando Sancho de Salas}
\address[F. S. de S.]{Departamento de Matem\'{a}ticas and 
Instituto Universitario de F\'isica Fundamental y Matem\'aticas (IUFFyM)\\
Universidad de Salamanca\\
Plaza de la Merced 1-4\\
E-37008 Salamanca, Spain}
\email{fsancho@usal.es}

\thanks{L.~A.~T. and A.~J.~L. were partially supported by [ED431C 2019/10] (Xunta de Galicia with E.U.'s FEDER funds). F.~S.~de~S. was partially supported by  Grant PID2021-128665NB-I00 funded by MCIN/AEI/ 10.13039/501100011033 and, as appropriate, by “ERDF A way of making Europe”.}


\date{\emph{released}, March 8, 2022, 
 \emph{typeset}: \today}

\theoremstyle{plain}
\newtheorem{thm}{Theorem}[section]
\newtheorem{lem}[thm]{Lemma}
\newtheorem{cor}[thm]{Corollary}
\newtheorem{prop}[thm]{Proposition}

\theoremstyle{remark}
\newtheorem*{rem}{Remark}

\theoremstyle{definition}

\newtheorem*{ex}{Example}
\newtheorem*{ack}{Acknowledgements}
\newtheorem{cosa}[thm]{}

\numberwithin{equation}{thm}

\newcommand{\CC}{{\mathcal C}}

\newcommand{\CE}{{\mathcal E}}
\newcommand{\CF}{{\mathcal F}}
\newcommand{\CG}{{\mathcal G}}
\newcommand{\CH}{{\mathcal H}}

\newcommand{\CK}{{\mathcal K}}
\newcommand{\CL}{\mathcal{L}}
\newcommand{\CM}{\mathcal{M}}
\newcommand{\CN}{\mathcal{N}}
\newcommand{\CO}{\mathcal{O}}

\newcommand{\CR}{\mathcal{R}}

\newcommand{\SA}{\mathsf{A}}

\newcommand{\SF}{\mathsf{F}}

\newcommand{\SP}{\mathsf{P}}

\newcommand{\SR}{\mathsf{R}}

\newcommand{\ST}{\boldsymbol{\mathsf{T}}}

\newcommand{\D}{\boldsymbol{\mathsf{D}}}

\newcommand{\LL}{\boldsymbol{\mathsf{L}}}
\newcommand{\R}{\boldsymbol{\mathsf{R}}}
\newcommand{\A}{\mathsf{A}}

\newcommand{\sch}{\mathsf {Sch}}

\newcommand{\cc}{\mathsf{c}}
\newcommand{\bb}{\mathsf{b}}
\newcommand{\da}{{\boldsymbol{\mathsf{a}}}}
\newcommand{\db}{{\boldsymbol{\mathsf{b}}}}
\newcommand{\dd}{{\boldsymbol{\mathsf{d}}}}

\newcommand{\per}{\mathsf{pf}}
\newcommand{\qc}{\mathsf{qc}}

\newcommand{\pc}{\mathsf{pc}}

\newcommand{\NN}{\mathbb{N}}
\newcommand{\ZZ}{\mathbb{Z}}

\newcommand{\ip}{{\mathfrak p}}
\newcommand{\iq}{{\mathfrak q}}
\newcommand{\im}{{\mathfrak m}}

\newcommand{\dirlim}[1]{\begin{array}[t]{c} {\rm lim}\\[-7.5 pt]
 {\longrightarrow} \\[-7.5 pt] {\scriptstyle {#1}} \end{array}}

\newcommand{\lto}{\longrightarrow}
\newcommand{\xto}{\xrightarrow}

\newcommand{\inc}{\hookrightarrow}
\DeclareMathOperator{\iso}{\tilde{\to}}

\DeclareMathOperator{\liso}{\tilde{\lto}}

\newcommand{\goesto}{\rightsquigarrow}

\newcommand{\imp}{\Rightarrow}
\newcommand{\dimp}{\Leftrightarrow}
\newcommand{\ldimp}{\Longleftrightarrow}

\DeclareMathOperator{\Hom}{Hom}
\DeclareMathOperator{\shom}{\CH\!\mathit{om}}

\DeclareMathOperator{\ext}{Ext}
\DeclareMathOperator{\tor}{Tor}

\DeclareMathOperator{\rhom}{\R{}Hom}

\DeclareMathOperator{\Img}{Im}
\DeclareMathOperator{\cok}{Coker}

\DeclareMathOperator{\spec}{Spec}

\DeclareMathOperator{\h}{H}
\DeclareMathOperator{\id}{id}
\DeclareMathOperator{\rank}{rank}

\DeclareMathOperator{\rGamma}{\R{}\varGamma\!}

\DeclareMathOperator{\perf}{\boldsymbol{\mathsf{Perf}}}
\newcommand{\Dbc}[1]{\operatorname{\D^{\bb}}(#1)_{\pc}}
\newcommand{\Dbq}[1]{\operatorname{\D^{\bb}_{\qc}}(#1)}

\newcommand{\<}{\mkern-1mu}
\renewcommand{\>}{\mkern1mu}
\newcommand{\va}[1]{\vspace{#1pt}}

\newcommand{\biE}[3]{{\operatorname{K_0}}({#2} \xto{#1} {#3})}

\newcommand{\pf}[1]{{#1}^{}_{\<\<{\star}}}
\newcommand{\pb}[1]{{#1}^{\star}}

\newcommand{\BIK}{\mathrm{K}_0}

\newcommand{\ie}{{\it i.e.~}}
\newcommand{\cfr}{{\it cf.~}}

\begin{document}

\begin{abstract}
 Let $f \colon X \to Y$ be a morphism of concentrated schemes. We characterize $f$-perfect complexes $\CE$ as those such that the functor $\CE \otimes^{\LL}_X \LL f^*-$ preserves bounded complexes. We prove, as a consequence, that a quasi-proper morphism takes relative perfect complexes into perfect ones. We obtain a generalized version of the semicontinuity theorem of dimension of cohomology and Grauert's base change of the fibers. Finally, a bivariant theory of the Grothendieck group of perfect complexes is developed. 
\end{abstract}

\maketitle
\tableofcontents

\section*{Introduction}
Perfect complexes play an outstanding role in algebraic geometry as the derived version of vector bundles. In this paper we propose a relative notion of perfect complex, study its basic properties and obtain useful geometric applications. Additionally, we present a bivariant theory based on them.

A first notion of a relative   complex was developed by Illusie under the guidance of Grothendieck and it was published in a series of expos\'es \cite{gcf, erg, finrel} corresponding to SGA6, the \emph{S\'eminaire de G\'eom\'etrie Alg\'ebrique du Bois Marie} of 1966-67. These papers already contain a relative notion of perfect complexes. The precise definition \cite[D\'efinition 4.1]{finrel} is very general but somewhat difficult to check in practice.

Another related proposal was given by Lieblich \cite{lieb} in order to develop a moduli space of complexes for a proper morphism of schemes. His notion has good properties, as shown in \cite[\href{https://stacks.math.columbia.edu/tag/0DI0}{Tag 0DI0}]{sp} \textit{et seq}. Lieblich's notion is restricted to flat morphisms of finite presentation.

In this paper we propose for pseudo-coherent complexes a notion of relative perfection with respect to arbitrary scheme maps. We use the notion of \emph{finite flat dimension relative to a base} defined in terms of preservation of boundedness. This notion comes from an idea already used in a research on integral functors and Fourier-Mukai transforms in a singular setting by Hern\'andez, L\'opez and Sancho \cite{fmtgs,rifs}. We define a relative perfect complex as a pseudo-coherent complex which satisfies this relative finite flat dimension condition.

Our notion of relative perfect complex in \ref{relperfdef} is a priori different from the one in SGA6. We prove in Corollary~\ref{agreement} that for pseudo-coherent morphisms both notions agree using Illusie's characterization \cite[Proposition 4.4]{finrel}. 

The main advantage of our approach is that it is flexible enough to give concise proofs of its expected properties. Further, it is possible to give nice characterizations of relative perfect complexes with respect to quasi-proper morphisms. As a consequence of the good behavior of this notion, we develop a bivariant theory based on relative perfect complexes.

One of the main results in this paper is Theorem \ref{absrelmor} that states for a quasi-proper map $f \colon X \to Y$ (definition \ref{qprop}) that the derived direct image functor carries $f$-perfect complexes into perfect complexes. This theorem generalizes \cite[\href{https://stacks.math.columbia.edu/tag/0DJT}{Tag 0DJT}]{sp} to the non flat situation and its proof is concise without making recourse to Noetherian approximation. Notice that a quasi-proper map between Noetherian schemes is just a proper map. The statement of Theorem \ref{absrelmor} may look innocent at first sight but, as a matter of fact, it has a lot of consequences. Paired with a subtler notion of derived fiber \eqref{niderfib} we give a broad generalization, covering the non Noetherian case, of the classical semicontinuity and Grauert's theorems on base change for quasi-compact and quasi-separated schemes.

\bigskip

Let us discuss in greater detail the contents of this paper. The first section sets the basic conventions and recalls some well-known facts. Along this paper a scheme means a quasi-compact and quasi-separated scheme (also called concentrated in the literature). For completeness, we include a direct proof of a criterion of boundedness in terms of derived homs from perfect complexes (Theorem \ref{biglem2011}).

In the next section \ref{sec3} we discuss (absolute) perfect complexes from the point of view of the functorial characterization of finite flat dimension. We give proofs for the following properties: perfect complexes are detected fiberwise and co-fiberwise (Theorem \ref{perfcar}), on a regular scheme a bounded pseudo-coherent complex is perfect (Corollary \ref{cor105}). We claim no originality about these results but we give leaner proofs. An alternate approach to these results is given in \cite[\href{https://stacks.math.columbia.edu/tag/068W}{Tag 068W}]{sp} \textit{et seq}. The last result, that we have not found in the literature, characterizes a perfect complex as one that tensoring with it commutes with products (Proposition \ref{comprod}).

Section \ref{sec4} introduces the notion of relative perfect complex using the functorial point of view.  We prove the main theorem: a quasi-proper morphism takes relative perfect complexes into perfect complexes.  We prove several invariance results and the local nature of relative perfection  (Propositions \ref{locchar0} and \ref{locchar}). We also treat the behavior of relative perfect complexes through transverse squares that will play an important role in the last section of the paper, and we characterize relative perfection for flat morphisms with regular fibers (Proposition \ref{perffib}).

Section \ref{sec6} presents some additional properties of $f$-perfect complexes when the morphism $f \colon X \to Y$ is quasi-proper. First, it is shown that our definition agrees with the one in SGA6 (Corollary \ref{agreement}). We prove that relative perfection may be checked fiberwise and characterize $f$-perfect complexes through its behavior with respect to $\R{}f_*$ (Theorem \ref{carrelperf}). Finally, we show that for a complex $\CE$ over $X$, $f$-perfection is characterized by the fact that the functor $\CE \otimes^{\LL} -$ commutes with products of sheaves coming form the base (Proposition \ref{comprodrel}) thus giving a relative version of Proposition \ref{comprod}.

Section \ref{sec5} is devoted to obtain some consequences from Theorem \ref{absrelmor}, namely, semicontinuity and Grauert type theorems. We prove that Euler-Poincar\'e characteristic is locally constant (Theorem \ref{locconst}), the semicontinuity of dimensions of cohomology spaces of the nice derived fibers (Theorem \ref{semicont}) and Grauert's theorem on base change (Theorem \ref{grauert}). These theorems vastly generalize the usual situation getting rid of Noetherian hypothesis and obtaining the aforementioned results for sheaves (and complexes) not necessarily flat over the base, see Example on page \pageref{ejemplo}.

In the last section (\S \ref{bibthy}), we apply all this machinery to show that the Grothendieck group of perfect complexes defines a bivariant theory in the sense of Fulton and MacPherson \cite{fmc}. A first version of this result was given by Pascual in \cite{PG}, in a more restrictive setup.

We have tried to make the paper as self-contained as possible. The basic reference for derived categories of quasi-coherent sheaves used in this paper is Lipman's book \cite{yellow}. 

\begin{ack}
 We thank A. Casta\~no, J. Lipman and A. Neeman for some exchanges regarding this paper. Discussing \cite{fmtgs}, Lipman asked the third author about the relationship between the complexes of finite homological dimension defined there and relative perfect complexes. This paper is an extended answer to this question. We thank the referee for a careful reading and comments that helped to improve the paper.
\end{ack}

\section{Preliminaries and notation}

\begin{cosa}
Throughout this paper schemes  are assumed to be  \emph{concentrated}, i.e.,   quasi-compact and quasi-separated. A quasi-compact open subscheme of a concentrated scheme is also concentrated. Let us recall that a map between concentrated schemes is itself concentrated \cite[ \S 6.1, pp. 290 \textit{ff}.]{GD}.

For such a scheme $X$ we will denote by $\SA(X)$ the category of $\CO_X$-Modules and by $\SA_\qc(X)$ its subcategory of \emph{quasi-coherent} $\CO_X$-Modules and their unbounded derived categories by $\D(X) : = \D(\A(X))$ and $\D(\A_\qc(X))$, respectively. We most frequently use the full subcategory $\D_\qc(X)$ of complexes of $\D(X)$ with cohomology in $\SA_\qc(X)$.  For brevity, we will write $\Hom_X(-,-)$ for $\Hom_{\D(X)}(-,-)$ for any scheme $X$.

We shall  widely use, without further reference, the following basic results for a  morphism of schemes $f\colon X\to Y$:

\begin{enumerate}[ 1.]
 \item  The functor $\R{}f_* \colon \D_\qc(X)\to \D_\qc(Y)$ is bounded, that is, if $\CM\in \D_\qc(X)$ is bounded, so is $\R{}f_*\CM$ (\cite[Proposition (3.9.2)]{yellow}).
 \item Projection formula (\cite[Proposition (3.9.4)]{yellow}): for any $\CM\in \D_\qc(X)$ and any $\CN\in \D_\qc(Y)$, the natural morphism
\[ 
\R{}f_* \CM \otimes^{\LL}_Y\CN\to \R{} f_*(\CM\otimes^{\LL}_X\LL f^*\CN)
\] is an isomorphism.
\end{enumerate}

Recall that a complex $\CE \in \D_\qc(X)$ is called \emph{strictly perfect} if it is a bounded complex of locally free finitely generated modules. The complex $\CE$ is \emph{perfect} if it is locally quasi-isomorphic to a strictly perfect complex. We denote the full subcategory of $\D_\qc(X)$ made of perfect complexes by $\D(X)_\per$. For convenience, throughout the paper, we will use the shorthand:
\[
\perf(X) := \D(X)_\per
\]
\end{cosa}

On a triangulated category $\ST$, an object $E$ is called \emph{compact} if the functor $\Hom_{\ST}(E,-)$ commutes with coproducts. On a  scheme $X$ a complex in $\D_\qc(X)$ is compact if and only if it is perfect by \cite[Theorem 3.1.1]{bb}. They also prove that there exists a single perfect generator in $\D_\qc(X)$.
 
\begin{rem}
The equivalence of perfect and compact complexes is essentially due to Thomason and Tro\-baugh (see \cite[Theorem 2.4.3]{tt}) but notice that the literal notion of compact object was established later, see \cite[Corollary 2.3.]{Ngd} where a proof is given for a  separated scheme $X$ assuming that $\D_\qc(X)$ is compactly generated. A slightly different proof under semi-separated hypothesis is given in \cite[Proposition 4.7.]{asht}.
\end{rem}

\begin{cosa}\label{perfdual}
For a complex $\CE \in \D(X)$ we will denote its dual as
\[
\CE^\vee := \R\!\shom^\bullet_X(\CE, \CO_X).
\]
If $\CE$ is perfect, then so is $\CE^\vee$ and there is a canonical isomorphism
\[
\R\!\shom^\bullet_X(\CN, \CE^\vee \otimes^{\LL}_X \CM) \cong
\R\!\shom^\bullet_X(\CN \otimes^{\LL}_X \CE, \CM) 
\]
for any $\CN, \CM \in \D(X)$. It is a local property, so we may assume that $\CE$ is a strictly perfect and in this case the isomorphism is obvious. This property can be expressed as saying that $\CE$ is strongly dualizable, \cfr \cite[\S 4.3]{asht}.
\end{cosa}

\begin{cosa}\label{pscoh}
A complex $\CE \in \D_\qc(X)$ is called \emph{pseudo-coherent} if every point $x \in X$ has an affine neighborhood $j \colon U \inc X$ such that $j^*\CE$ is quasi-isomorphic to a bounded-above complex of finite-rank free $\CO_U$-Modules \cite[D\'efinition 2.3]{gcf}. We denote the full subcategory of pseudo-coherent complexes in $\D_\qc(X)$ by $\D(X)_{\pc}$. If $X$ is a Noetherian scheme, a pseudo-coherent complex is precisely a bounded-above complex with coherent cohomology by \cite[Corollaire 3.5 b)]{gcf}, in other words, $\D(X)_{\pc} = \D^-_{\cc}(X)$. See \cite[\S 4.3]{yellow} for further discussion about this property.

We will denote by $\Dbc{X}$ the full subcategory of $\D_\qc(X)$ formed by  pseudo-coherent complexes with bounded cohomology. If $X$ is Noetherian then $\Dbc{X}$ is $\D^{\bb}_{\cc}(X)$, the category of  complexes with bounded and coherent cohomology.
\end{cosa}

\begin{cosa}\label{finflat}
 Let $f \colon X \to Y$ be a morphism of schemes. We say that $\CE \in \D_\qc(X)$ has \emph{finite flat dimension for} $f$ if $\CE \otimes^{\LL}_X \LL f^* \CM$ is bounded for every $\CM \in \Dbq{Y}$. This notion is weaker than that of finite flat $f$-amplitude in \cite[D\'efinition 3.1]{finrel}. If $\CO_X$ has finite flat dimension for $f$, then we say that $f$ has \emph{finite flat dimension}. A complex $\CE \in \D_\qc(X)$ has \emph{finite flat dimension} if $\CE \otimes^{\LL}_X  \CM$ is bounded for every $\CM \in \Dbq{X}$, that is, if $\CE$ has finite flat dimension for $\id_X$.
\end{cosa}

\begin{cosa} 
\textbf{Pseudo-coherent morphisms.} A morphism $f \colon X \to Y$ is \emph{pseudo-coherent} if for every point of $X$ there is an open neighborhood $U \subset X$, a smooth morphism $p \colon P \to Y$, and a closed immersion $i \colon U \to P$ such that $f|_U = p \circ i$ where $i_*\CO_U$ is pseudo-coherent on $P$ (see \ref{pscoh}).  This definition essentially agrees with \cite[D\'efinition 1.2, p. 228]{finrel} in view of \cite[Proposition 4.4, p. 252]{finrel}. See also \cite[\href{https://stacks.math.columbia.edu/tag/067X}{Tag 067X}]{sp}.
 
Notice that if $f$ is a finite-type morphism with $Y$ Noetherian then $f$ is pseudo-coherent. Smooth morphisms and regular immersions are pseudo-coherent. Also, any composition of pseudo-coherent maps is still pseudo-coherent \cite[Corollaire 1.14, p. 236]{finrel}.
\end{cosa}


\begin{cosa} \label{qprop}
\textbf{ Quasi-proper morphisms.} 
Following \cite[Definition (4.3.3.1)]{yellow}, a  morphism $f \colon X \to Y$  is \emph{quasi-proper} if $\R{}f_*$ preserves pseudo-coherence, i.e. it takes $\D(X)_{\pc}$ into $\D(Y)_{\pc}$. An important instance of this notion is given by the fact that \emph{proper pseudo-coherent morphisms are quasi-proper} \cite[Corollary (4.3.3.2)]{yellow}. This result relies heavily in Kiehl’s Finiteness Theorem \cite[Theorem $2.9'$, p. 315]{kie}. In particular,  a proper morphism with $Y$ Noetherian is automatically quasi-proper.
\end{cosa}

\begin{cosa}\textbf{A criterion for boundedness}.
The characterization of bounded complexes in Proposition \ref{lem2011} will not be used until section \ref{sec6}. We advise the reader to skip it on first reading. Our proof is essentially self contained modulo standard facts and Thomason-Neeman localization theorem \cite{Ntty}, that we use as it is stated in \cite{Ngd}. For an alternative treatment, see \cite[\href{https://stacks.math.columbia.edu/tag/0GEQ}{Tag 0GEQ} and \href{https://stacks.math.columbia.edu/tag/0GEN}{Tag 0GEN}]{sp}.

Let $V = \spec(R)$ be an affine scheme, $Z\subset V$ a closed subset determined by the ideal $\langle t_1, \dots , t_r\rangle \subset R$. Let $W := V\setminus Z$ and $j \colon W \to V$ be the canonical open embedding. Denote by $\mathbf{U} = \{D(t_i) ; i \in \{1, \dots, r\}\}$ the open covering of $W$ given by the generators of the ideal defining $Z$. As $j$ is separated, we can compute $\R{}j_*$ by applying $j_*$ to the \v{C}ech complex relative to this covering \cite[chapter III, Lemma 4.2]{HAG}. Let $\CC(\mathbf{t}) := j_*\CC(\mathbf{U}, \CO_W)$. For $\CN \in \D_\qc(V)$, there is a canonical isomorphism
\begin{equation}\label{compCzech}
 \R{}j_*  j^*\CN  \cong \CC(\mathbf{t}) \otimes^{\LL}_V \CN
\end{equation}

There is a close relation of $\CC(\mathbf{t})$ with Koszul complexes.
Let  $\CK(\mathbf{t})$ be the Koszul complex associated to the sequence $\mathbf{t}= t_1, \ldots , t_r$. For each integer $s\geq 1$, let  $\mathbf{t}^s$ be the sequence $\mathbf{t}^s= t_1^s, \ldots , t_r^s$. The complex $\CK(\mathbf{t}^s)$ is a perfect complex concentrated in degrees $-r$ and $0$, and its dual 
\[
\CK(\mathbf{t}^s)^{\vee} :=\shom^\bullet_V (\CK(\mathbf{t}^s), \CO_V),
\]
is canonically  isomorphic to the perfect complexes $\CK_\vee(\mathbf{t}^s)$ (that we describe next  in \ref{notacion}) and $\CK(\mathbf{t}^s)[-r]$ (this is the so called self-duality of the Koszul complex \cite[Proposition 17.15.]{Eca}).
%
For any $t\in \Gamma(V,\CO_V)$, let $\CK_\vee (t)$ be the complex concentrated in degrees $0$ and $1$
\[
\cdots \rightarrow 0 \rightarrow \CO_V \rightarrow  \CO_V\rightarrow 0 \rightarrow \cdots,
\]
which is multiplication by $t$ from $\CK_{\vee}^0(t)= \CO_V$ to $\CK_{\vee}^1(t)= \CO_V$. 
For integers $1\leq s\leq s'$, let us consider the map  of complexes $\CK_{\vee}(t^{s})\rightarrow \CK_{\vee}(t^{s'})$ which is the identity in degree $0$ and multiplication by $t^{s'-s}$ in degree $1$. Let us denote by  $\CK_{\infty}(t)$ the direct limit of the direct system of complexes $\{\CK_{\vee}(t^s)\}_{s\in \NN}$. Then $\CK_{\infty}(t)$ is a complex concetrated  in degrees $0$ and $1$ where it looks like the canonical map $\CO_V\rightarrow {j_t}_* \CO_{D(t)}$, being  $j_t\colon D(t)\rightarrow V $  the canonical embedding of the principal open subset $D(t) \subset V$.
That is, $\CK_{\infty}({t})$ coincides with the augmented complex $\,\CO_V\to \CC({t})$. Set
\begin{equation} \label{notacion}
\CK_{\vee}(\mathbf{t})= \CK_{\vee}({t_1})\otimes_{\CO_V}\cdots\otimes_{\CO_V}  \CK_{\vee}( {t_r}),
\end{equation} 
and let us denote by $\CK_{\infty}(\mathbf{t})$ the complex\footnote{notation $\CK_{\infty}(\mathbf{t})$ as in \cite[\S 3.1]{AJL1}, where $\CK_{\vee}(\mathbf{t}^s)$ is denoted by $\CK(\mathbf{t}^s)$.}
\[
\CK_{\infty}(\mathbf{t}) := \dirlim{s\in \NN}  \CK_{\vee}(\mathbf{t}^s)= \CK_{\infty}({t_1})\otimes_{\CO_V}\cdots\otimes_{\CO_V}  \CK_{\infty}( {t_r}).
\]
Hence, there is a canonical distinguished triangle
\begin{equation}
\label{TKCz}
 \CK_{\infty}(\mathbf{t}) \lto \CO_V \lto \CC(\mathbf{t}) \stackrel{+}{\lto}.
\end{equation}
%
Let $\CN\in \D_\qc(V)$. Combining \eqref{compCzech} and \eqref{TKCz}, in view of the canonical triangle
\[
\rGamma_Z (\CN) \lto  \CN \lto \R{}j_*j^* \CN \stackrel{+}{\lto}
\]
we may identify 
\[
\rGamma_Z (\CN) \cong \CK_{\infty}(\mathbf{t}) \otimes_{\CO_V} \CN,
\]
because everything in sight commutes with (derived) tensor products. Therefore there exists a canonical isomorphism 
\[
\rGamma_Z (\CN) \cong
\dirlim{s\in \NN}  \CK_\vee(\mathbf{t}^s) \otimes_{\CO_V} \CN
\cong
\dirlim{s\in \NN}  \shom^\bullet_V ( \CK(\mathbf{t}^s), \CN).
\]

\end{cosa}

\begin{lem} \label{casoafin}
In the previous setting, if $\CN\in \D_\qc(V)$ and $i\in \ZZ$ are such that 
\[
\h^i\rhom^\bullet_V(\CK(\mathbf{t}), \CN)=0
\]   
then $\h^i\rGamma_Z(\CN)=0$.
\end{lem}

\begin{proof}
For any two elements $t,\,t'\in R$, the Koszul complex $\CK(t \hspace{-1pt}\cdot \hspace{-1pt}t')$ is homotopy equivalent to the mapping cone of the morphism of complexes 
\[
\CK(t)[-1] \stackrel{\varphi}{\lto} \CK(t')
\]
determined by the identity map in degree $0$. As a consequence, for any $a, b \in \NN$ and $l \in \{1,\ldots,r\}$  there exist distinguished triangles
\[
\CK(t_1, \dots, t_l^{a}, \dots, t_r) \lto \CK(t_1, \dots, t_l^{a+b}\!, \dots, t_r) \lto \CK(t_1, \dots, t_l^{b}, \dots, t_r) \stackrel{+}{\lto}
\]
Let us assume that  $\h^i\rhom^\bullet_V(\CK(\mathbf{t}), \CN)=0$. Let us fix $b=1$, then inductively starting with $a=1$ and $l=1$, and applying $\rhom^\bullet_V(-, \CN)$ to the corresponding triangles we conclude that
$
\h^i\rhom^\bullet_V(\CK(\mathbf{t}^s), \CN)=0, 
$
for any $s\in \NN$. The complex $\shom^\bullet_V(\CK(\mathbf{t}^s), \CN)$, of quasi-coherent sheaves, corresponds to the complex of modules $\Hom^\bullet_V(\CK(\mathbf{t}^s), \CN)$. Therefore
\[
\h^i\rGamma_Z (\CN) \cong \dirlim{s\in \NN}\!\h^i\shom^\bullet_V(\CK(\mathbf{t}^s), \CN) = 0.\qedhere
\]
\end{proof}

\begin{prop} \label{lem2011}
Let $X$ be a  scheme. For any $\CM\in \D_\qc(X)$ the following statements are equivalent:
\begin{enumerate}
\item \label{1bounded}The complex $\CM$ is bounded.
\item \label{2bounded}For any perfect complex $\CE\in \D_\qc(X)$, $\rhom^\bullet_X(\CE,\CM)\in \D^\bb(\ZZ)$.
\end{enumerate}
\end{prop}

\begin{proof} For any complex $\CE\in \perf(X)$, by \ref{perfdual}, there is a natural isomorphism
\begin{equation}
\rhom^\bullet_X(\CE,\CM)\cong \R\Gamma(X, \CE^{\vee}\!\!\otimes_X^{\LL}\CM) 
\label{caso1}.
\end{equation}
As a consequence, the functor
$
\rhom^\bullet_X(\CE,-)\colon \D_\qc(X) \to \D(\ZZ)
$
is bounded, so (\ref{1bounded}) implies (\ref{2bounded}). 

To prove  (\ref{2bounded}) implies (\ref{1bounded}) we proceed by induction on $n$ the number of elements in an affine open covering of $X$. Let us assume $X=V_1\cup \cdots \cup V_{n-1}\cup V_n$ is an affine open covering, and $\CM\in \D_\qc(X)$ is a complex satisfying (\ref{2bounded}).

When $n=1$, $X$ is affine. By (\ref{caso1}) and the property (\ref{2bounded}) for $\CE=\CO_X$, 
the complex $\R\Gamma(X, \CM)$ is bounded, equivalently, $\CM\cong \R\Gamma(X, \CM)^{\widetilde{\quad}}$ is bounded.

In case $n>1$, let $U:=V_1\cup \cdots \cup V_{n-1}$, $V:=V_n$ and $Z:=X\setminus U=V\setminus U$, and let $j\colon V\hookrightarrow X$  and $j_U\colon U\hookrightarrow X$ be the corresponding open embeddings.

We start by proving that $\rGamma_Z \CM$ is bounded. Indeed, let $R:=\Gamma (V, \CO_X)$ and  
$t_1, \dots , t_r\in R$ be a sequence such that the ideal $\langle t_1, \dots , t_r\rangle \subset R$ determines the closed subset $Z\subset V=\spec(R)$. The perfect complex $\CK(\mathbf{t})\in \D_\qc(V)$ is supported on $Z$ that is also a closed subset of $X$, therefore $\R j_*\CK(\mathbf{t})\in \D_\qc(X)$ is perfect and $\R j_*\CK(\mathbf{t})=j_!\CK(\mathbf{t})$, in particular  
\[
\rhom^\bullet_V(\CK(\mathbf{t}), j^*\CM) \cong \rhom^\bullet_X(\R j_*\CK(\mathbf{t}), \CM);
\]
then $\rhom^\bullet_V(\CK(\mathbf{t}), j^*\CM)$ is bounded by property (\ref{2bounded}) of $\CM$. From Lemma~\ref{casoafin},  the complex $\rGamma_{Z} (j^*\CM)$ is bounded, hence  $\R j_* \rGamma_Z (j^*\CM)$  is bounded. The canonical map $\rGamma_Z (\CM) \to \rGamma_Z (\R j_* j^*\CM) $ is an isomorphism, because $Z\subset V$ is a closed subset of $X$. Using the canonical isomorphism 
$
\rGamma_Z (\R j_* j^*\CM)\cong \R j_* \rGamma_{Z} (j^*\CM), 
$
we conclude that  $\rGamma_Z (\CM)$ is bounded.

The functor $\R j_{U *}\colon \D_\qc(U)\to \D_\qc(X) $ is bounded and the complex $\rGamma_Z (\CM)$ is bounded, therefore the distinguished triangle
\begin{equation} \label{otro}
\rGamma_Z \CM\lto \CM
\lto \R j_{U *} (\CM\vert_{U})\stackrel{+}{\lto} 
\end{equation}
exhibits that $\CM$ is bounded if and only if $\CM\vert_{U}$ is bounded. Being $U:= V_1\cup \cdots \cup V_{n-1}$ an affine covering, by inductive hypothesis, to prove that $\CM\vert_{U}$ is bounded is equivalent to prove $\rhom^\bullet_U(\CF,\CM\vert_{U})$ is bounded, for any perfect complex $\CF\in \D_\qc(U)$.
By \cite[Theorem~2.1,~(2.1.4)]{Ngd}, given a perfect complex $\CF\in \D_\qc(U)$ there exists a perfect complex $\CE\in \D_\qc(X)$ such that $j_U^*\CE\cong \CF\oplus \CF[1]$. Therefore, to check $\rhom^\bullet_U(\CF,\CM\vert_{U})$ is bounded, for any perfect complex $\CF\in \D_\qc(U)$, is equivalent to check that the complex
\[
\rhom^\bullet_U(j_U^*\CE, \CM\vert_{U}) \cong \rhom^\bullet_X(\CE, \R j_{U *} \CM\vert_{U})
\]
 is bounded for any  perfect complex $\CE\in \D_\qc(X)$. For such an $\CE\in \D_\qc(X)$, applying the bounded functor  $\rhom^\bullet_X(\CE, -)$ to the triangle (\ref{otro}) we get the distinguished triangle 
\begin{equation*} 
\rhom^\bullet_X(\CE,\rGamma_Z \CM) \lto \rhom^\bullet_X(\CE,\CM)
\lto \rhom^\bullet_X(\CE,\R j_{U *} (\CM\vert_{U}))\stackrel{+}{\lto},
\end{equation*}
where $\rhom^\bullet_X(\CE,\rGamma_Z \CM)$ is bounded by the previous discussion and the complex $\rhom^\bullet_X(\CE, \CM)$ is bounded by hypothesis (\ref{2bounded}). As a consequence, the complex $\rhom^\bullet_X(\CE,\R j_{U *} (\CM\vert_{U}))$ is bounded, as wanted.
\end{proof}

\begin{thm}\label{biglem2011}
Let $X$ be a  scheme. For any $\CM\in \D_\qc(X)$ the following are equivalent:
\begin{enumerate}
\item \label{1bisbounded}The complex $\CM$ is bounded.
\item \label{2bisbounded}For a perfect generator $\CL$ of \,$\D_\qc(X)$, $\rhom^\bullet_X(\CL,\CM)\in \D^\bb(\ZZ)$.
\end{enumerate} 
\end{thm}

\begin{proof}
 The fact that (\ref{1bisbounded}) implies (\ref{2bisbounded}) follows from (\ref{1bounded}) $\imp$ (\ref{2bounded}) in Proposition \ref{lem2011}. 
 
For the converse, let $\SP$ be the full category of $\D_\qc(X)$  whose objects are those $\CE \in \perf(X)$ such that $\rhom^\bullet_X(\CL,\CM)\in \D^\bb(\ZZ)$. Note that $\SP$ is a thick subcategory of $\D_\qc(X)$ and $\CL$ belongs to $\SP$ by (\ref{2bisbounded}). Now, using \cite[Lemma 2.2.]{Ntty} all compact objects in $\D_\qc(X)$ belong to $\SP$. By  \cite[Theorem 3.1.1]{bb}, the compact objects in $\D_\qc(X)$ are precisely the perfect complexes, so (\ref{1bisbounded}) follows.
\end{proof}

\begin{rem} If $f\colon X\to Y$ is a projective morphism and $\CO_X(1)$ is a relatively very ample invertible sheaf on $X$, then by Theorem \ref{biglem2011} an object $\CM\in \D_\qc(X)$ is bounded if and only if $\R{}f_*(\CM\otimes_X\CO_X(r))$ is bounded for any $r\in\ZZ$. Indeed, we may assume that $Y$ is affine, and then in this case the family $\{\CO_X(r) \,/ \, r \in \ZZ \}$ generates $\D_\qc(X)$. The claim follows using the isomorphism $\rhom^\bullet_X(\CO_X(-r),\CM) \cong \R{}f_*(\CM\otimes_X\CO_X(r))$.

Alternatively, assuming again that $Y$ is affine the result follows from \cite[Lemma 1.4]{fmtgs} (where the Noetherian hypothesis on $Y$ is unnecessary).
\end{rem}

\section{Characterizing perfect complexes}\label{sec3}

\begin{cosa} \label{notres}
Let $X$ be a scheme. For each $x \in X$, let $\kappa(x)$ denote its  residue field and let $i_x \colon \spec(\kappa(x)) \to X$ be the canonical morphism. From now on, we will denote $\CK(x) : = i_{x *}\widetilde{\kappa(x)}$. Notice that the functor $i_{x *}$ is exact, being affine, so there is no need to derive it. 
\end{cosa}

The next lemma is well-known. We include the proof here for convenience.

\begin{lem}\label{fintorext}
 Let $(R, \im, k)$ denote a local ring, $\im$ its maximal ideal and $k:= R/\im$ its residue field. Let $M$ be a finitely presented $R$-module. The following are equivalent: 
\begin{enumerate}
 \item The module $M$ is free.
 \item It holds that $\tor_1^R(M,k) = 0$.
 \item It holds that $\ext^1_R(M,k) = 0$.
\end{enumerate}
\end{lem}

\begin{proof}
Choosing a minimal system of generators of $M$, we define an epimorphism $u \colon R^n \to M$ such that $u \otimes_R k$ is an isomorphism. The module $N := \ker(u)$ is finitely generated. If $\tor_1^R(M,k) = 0$, then $N \otimes_R k = N/\im N = 0$, which implies that $N = 0$ by Nakayama's lemma. Analogously, if $\ext^1_R(M,k) = 0$, then $\Hom_R(N/\im N ,k) = 0$ and again $N = 0$ by Nakayama's lemma.
\end{proof}

We have the following well-known characterizations of perfect complexes.

\begin{thm}\label{perfcar}
Let $X$ be a   scheme and $\CE \in \Dbc{X}$. The following conditions are equivalent:
\begin{enumerate}
 \item\label{uno} The complex $\CE$ is perfect.
 \item\label{dos} The complex $\CE$ has finite flat dimension (\ref{finflat}).
 \item\label{tres} The complex $\CE \otimes^{\LL}_X \CK(x)$ is bounded, for every $x \in X$.
 \item\label{cuatro} The complex $\LL{}i_x^* \CE$ is bounded, for every $x \in X$.
 \item\label{cinco} The complex $\R\!\shom_X^\bullet(\CE, \CM)$ is bounded, for every $\CM \in  \Dbq{X}$.
 \item\label{seis} The complex $\R\!\shom_X^\bullet(\CE, \CK(x))$ is bounded, for every $x \in X$.
\end{enumerate}
\end{thm}

\begin{proof}
 All the properties are local, therefore we may assume that $X$ is an affine scheme, $X = \spec(A)$. The implications \eqref{uno} $\imp$ \eqref{dos} $\imp$ \eqref{tres} are clear.
 
The equivalence \eqref{tres} $\dimp$ \eqref{cuatro} is a consequence of the isomorphism
\[
i_{x *}\LL{}i_x^* \CE \cong \CE \otimes^{\LL}_X \!\CK(x)
\]
having in mind that $i_x \colon \spec(\kappa(x)) \to X$ is affine.

Let us treat now \eqref{tres} $\imp$ \eqref{uno}. Let $E = \R\Gamma(X,\CE)$ and $L$ be a resolution of $E$ by a bounded above complex of finite free modules. Let $K_n := \ker(L^{n+1} \to L^{n+2})$. Since $E$ and $E \otimes^{\LL}_A \kappa(x)$ are bounded, there is a small enough $n$ such that
\begin{enumerate}[(a)]
 \item the complex $0 \to K_n \to L^{n+1} \to L^{n+2} \to \cdots$ is quasi-isomorphic to $E$,
 \item the sequence $\dots \to L^{n-1} \to L^{n} \to K_n \to 0$ is exact, and
 \item $\h^{n-1}(E \otimes^{\LL}_A \kappa(x)) = 0$.
 \end{enumerate}
 Then
 \[
 0 \overset{\text{(c)}}= \h^{n-1}(E \otimes^{\LL}_A \kappa(x)) 
 = \h^{n-1}(L \otimes_A \kappa(x)) 
 \overset{\text{(b)}}= \tor_1^A(K_n,\kappa(x)).
 \]
By Lemma \ref{fintorext} , $K_n$ is free in a neighborhood of $x$. Using (a) we see that $\CE$ is a perfect complex on this neighborhood.

An analogous argument shows that if $\R\!\shom_X^\bullet(\CE, \CK(x))$ is bounded then it follows that $\ext^1_A(K_n, \kappa(x)) = 0$ for $n \ll 0$ and, similarly one deduces that $K_n$ is free in a neighborhood of $x$. This yields \eqref{seis} $\imp$ \eqref{uno}.

Finally, the implications \eqref{uno} $\imp$ \eqref{cinco} $\imp$ \eqref{seis} are clear.
\end{proof}

\begin{rem}
In the Noetherian case this result is proved in \cite[Lemma 1.2]{fmtgs}. In the affine case the result corresponds to \cite[\href{https://stacks.math.columbia.edu/tag/068W}{Tag 068W}]{sp}.
\end{rem}

\begin{cosa}
 Recall that a regular scheme $X$ is a Noetherian scheme such that $\CO_{X,x}$ is a regular local ring for every $x \in X$. It is a celebrated theorem of Auslander-Buchs\-baum-Serre that a local ring is regular if, and only if, all its modules have finite projective dimension which in turn, it is equivalent to the finiteness of the flat dimension of its residue field, see \cite[Theorem 4.4.16]{wha}.
\end{cosa}

The following well-known result is an easy consequence of Theorem \ref{perfcar}.
\begin{cor}\label{cor105}
Let $X$ be a regular scheme. It holds that
\[
\CE \in \Dbc{X} \quad \ldimp \quad \CE \in \perf(X). 
\]
\end{cor}

\begin{proof}
Perfect complexes are bounded and pseudo-coherent. Conversely, if $\CE \in \Dbc{X}$, then $\CE \otimes^{\LL}_X \CK(x)$ is bounded because $\kappa(x)$ is of finite flat dimension as $\CO_{X,x}$-module for every $x \in X$. Then $\CE$ is perfect by Theorem \ref{perfcar}.
\end{proof}

\begin{cosa} \label{prodqc}
Let $X$ be a  scheme and  $\iota \colon  \D_\qc(X) \to \D(X)$ the canonical embedding. By \cite[Theorem 3.1.1]{bb}, $\D_\qc(X)$ has a perfect generator. Applying \cite[Theorem 4.1]{Ngd}, as $\iota$ commutes with coproducts, it possesses a right adjoint. As a consequence, the category $\D_\qc(X)$ has all products. \emph{Caveat lector}, the functor $\iota$ does \emph{not} preserve products.
\end{cosa}

\begin{prop}\label{comprod}
Let $X$ be a  scheme and $\CE \in \D_\qc(X)$. The complex $\CE$ is perfect if, and only if, the functor $\CE \otimes^{\LL}_X - \colon \D_\qc(X) \to \D_\qc(X)$ commutes with products. 
\end{prop}

\begin{proof}
 If $\CE \in \perf(X)$ then the functor $\CE \otimes^{\LL}_X -$ has a left adjoint $\CE^{\vee} \otimes^{\LL}_X - $, therefore it commutes with products (see \ref{perfdual}).

Let us check the converse. The functor $\SF$ from $\D_\qc(X)$ to abelian groups defined by $\SF(-)=\Hom_X(\CO_X,\, \CE \otimes^{\LL}_X -)$ commutes with products by the hypothesis on $\CE$. It follows that it is representable by \cite[Theorem 8.6.1.]{Ntc} or \cite[Theorem B]{Kbr} because $\D_\qc(X)$ is compactly generated \cite[Theorem 3.1.1]{bb}. Let $\CR \in \D_\qc(X)$ be the representing object, \ie
\begin{equation}\label{repr}
 \Hom_X(\CO_X,\, \CE \otimes^{\LL}_X \CM) \cong \Hom_X(\CR, \CM) \qquad (\CM \in \D_\qc(X)).
\end{equation}
Since $\SF$ commutes with coproducts, $\CR$ is compact, and therefore perfect. The identity of $\CR$ gives, through \eqref{repr}, a morphism $\CO_X \to \CE \otimes^{\LL}_X \CR$ and then a morphism $\varphi \colon \CR^{\vee} \to \CE$. We conclude by proving that $\varphi$ is an isomorphism. Again, since $\D_\qc(X)$ is compactly generated, it suffices to see that the morphism induced by $\varphi$,
\[
\Hom_X(\CL,\CR^{\vee}) \lto \Hom_X(\CL,\CE), 
\]
is an isomorphism for every $\CL \in \perf(X)$. Indeed, this map factors as the following chain of isomorphisms
\[
\Hom_X(\CL,\CR^{\vee}) \iso \Hom_X(\CR,\CL^{\vee}) \iso 
\Hom_X(\CO_X, \CE \otimes^{\LL}_X \CL^{\vee}) \iso  \Hom_X(\CL,\CE).
\]
The first one is isomorphism by adjunction, the middle one by \eqref{repr} and the last one because $\CL$ is perfect and, thus, strongly dualizable as in \eqref{perfdual}.
\end{proof}


\section{Relatively perfect complexes}\label{sec4}

\begin{cosa}\label{relperfdef}
Let $f \colon X \to Y$ be a morphism of schemes. A complex $\CE \in \D(X)_{\pc}$ is \emph{relatively perfect} over $Y$ or just \emph{$f$-perfect} if 
\[
\CE \otimes^{\LL}_X \LL f^* \!\CM  \,\, \text{ is bounded for every } \,\, \CM \in \Dbq{Y}.
\]
In \cite[Definition 1.3]{fmtgs}, these complexes are denominated of finite homological dimension over $Y$. By taking $\CM=\CO_Y$, we see that $\CE$ is bounded.

 We denote by $\perf(X | Y)$ or simply $\perf(f)$ the full subcategory of $\Dbc{X}$ formed by $f$-perfect complexes. 
 \end{cosa}

\begin{rem}
\begin{enumerate}[(a)]
 \item If $X = Y$ and $f = \id_X$ then, by \eqref{uno} $\dimp$ \eqref{dos} in Theorem \ref{perfcar}, $\id_X$-perfect is just (absolutely) perfect, so the terminology is justified.
 \item \label{operfffd} The structure sheaf $\CO_X$ is $f$-perfect if, and only if, $f \colon X \to Y$ is of finite flat dimension (\cfr \ref{finflat}).  
\end{enumerate}
\end{rem}

 The following Lemma is immediate.

\begin{lem}\label{tenper}
 Let $f \colon X \to Y$ be a morphism of schemes. If  $\CE$ is $f$-perfect and $\CL$ is perfect over $X$, then $\CL \otimes^{\LL}_{X} \CE$ is $f$-perfect. Hence, if $f$ is of finite flat dimension,  any perfect complex on $X$ is $f$-perfect.
\end{lem}

\begin{rem}
 For flat and locally of finite presentation morphism, the analogous result for Lieblich's notion is given in \cite[\href{https://stacks.math.columbia.edu/tag/0DI4}{Tag 0DI4}]{sp}.
\end{rem}

The following theorem is one of our main motivations for this work. It provides a criterion for relative perfection \emph{without} flatness.

\begin{thm}\label{absrelmor} Let $f \colon X \to Y$ be a  morphism of schemes  and let $\CE\in \D(X)_{\pc}$.

\begin{enumerate}[1.]
 \item\label{f-perf1} If $\CE$ is $f$-perfect, then $\R{}f_*\CE$ is of finite flat dimension. The converse also holds if $f$ is affine.
 \item\label{f-perf2} If $f$ is quasi-proper and $\CE$ is $f$-perfect, then $\R{}f_*\CE$ is perfect.
 \end{enumerate}
\end{thm}

\begin{proof}
\ref{f-perf1}. Let $\CM \in \Dbq{Y}$. By the projection formula, 
\[
\R{}f_*(\CE \otimes^{\LL}_X \LL f^* \CM) \cong \R{}f_*\CE \otimes^{\LL}_Y \CM.
\]
Hence, if $\CE$ is $f$-perfect, we conclude that  $\R{}f_*\CE \otimes^{\LL}_Y \CM$ is bounded because $\R{}f_*$ is bounded over $\D_\qc(X)$. Therefore, $\R{}f_*\CE$ is of finite flat dimension. If $f$ is affine, the boundedness of $\R{}f_*\CE \otimes^{\LL}_Y \CM$ is equivalent to the boundedness of $\CE \otimes^{\LL}_X \LL f^* \CM$, hence the converse also holds. 

\ref{f-perf2}. If $f$ is quasi-proper and $\CE$ is $f$-perfect, then $\R{}f_*\CE$ is of finite flat dimension and pseudo-coherent. By  Theorem \ref{perfcar}, $\R{}f_*\CE$ is perfect.
\end{proof}

\begin{cor}\label{corLN}
With the assumptions of the previous theorem, if we assume further that $f$ is of finite flat dimension, then any perfect complex \,$\CE$ on $X$ is $f$-perfect, hence  $\R{}f_*\CE$ is of finite flat dimension. If moreover $f$ is quasi-proper then $\R{}f_*\CE$ is perfect.
\end{cor}

The next result is a relative version of the previous Theorem. It will be used on section \ref{bibthy} as a basic building block of the bivariant theory.

\begin{prop}\label{pfper}
 Let $f \colon X \to Y$ and $g \colon Y \to S$ be 
 morphisms with $f$ quasi-proper. If $\CE \in \perf(g \circ f)$, then $\R{}f_* \CE \in \perf(g)$.
\end{prop}

\begin{proof}
First, $\R{}f_*\CE$ is  pseudo-coherent  because $f$ is quasi-proper (\ref{qprop}). Now, let $\CN \in \Dbq{S}$. As $\CE \in \perf(g \circ f)$ it follows that $\CE \otimes^{\LL}_X \LL (g \circ f)^* \CN$ is bounded.  Then $\R{}f_*(\CE \otimes^{\LL}_X \LL f^* \LL g^*  \CN) \cong \R{}f_*\CE \otimes^{\LL}_Y \LL g^*\CN$ is bounded, which means that $\R{}f_*\CE$ is $g$-perfect.
\end{proof}

\begin{cosa} Let $X$ be a  scheme, $U \subset X$ an open subscheme and $j \colon U \inc X$ the canonical embedding. For a given $\CE \in \D_\qc(X)$, we denote
 \[
 \CE_U := \R{}j_*  (\CE|_U)
 \]
 Note that $\CE_U \in \D_\qc(X)$ if $U$ is quasi-compact. Let $f \colon X \to Y$ be a morphism of schemes, and $f|_U \colon U \to Y$ the restriction of $f$ to $U$, \ie $f|_U = f \circ j$. Notice that $$\R{}f_*\CE_U\cong \R{}{f|_U}_*\CE|_U.$$
 \end{cosa}

\begin{prop}\label{prop204}
Let $f \colon X \to Y$ and $p \colon X' \to X$ be morphisms of schemes with   $p$ of finite flat dimension. If $\CE$ is $f$-perfect, then $\LL p^*\CE$ is  $(f \circ p)$-perfect. In particular, for any quasi-compact open subscheme $U$ of $X$, $\CE|_U$ is $f|_U$-perfect and then $\R{}f_*\CE_U$ is of finite flat dimension.
\end{prop}

\begin{proof}
Let $\CM \in \Dbq{Y}$. We have that
\[
\LL p^*\CE \otimes^{\LL}_{X'} \LL (f \circ p)^* \CM \cong
\LL p^*(\CE \otimes^{\LL}_{X} \LL f^* \CM).
\]
Now, $\CE \otimes^{\LL}_{X} \LL f^* \CM$ is bounded because $\CE$ is $f$-perfect and keeps being bounded after applying $\LL p^*$ because $p$ is of finite flat dimension. The final statement for a quasi-compact open subscheme $U$ follows from Theorem \ref{absrelmor},(\ref{f-perf1}).
\end{proof}

In the next propositions we show that relative perfection is   local on the base and on the source.

\begin{prop}\label{locchar0}
Let $f \colon X \to Y$ be a morphism of schemes and $\CE \in \D_\qc(X)$. Let $Y = V_1 \cup \dots \cup V_n$ with each $V_i$ a quasi-compact open subscheme of $Y$ and let $f_i \colon f^{-1}V_i \to V_i$ be the maps induced by $f$ for each $i \in \{1, \dots, n\}$. Then $\CE$ is $f$-perfect if, and only if, $\CE|_{f^{-1}V_i}$ is $f_i$-perfect for every $i \in \{1, \dots, n\}$.
\end{prop}

\begin{proof}
Let $\CM \in \Dbq{Y}$. If $\CE|_{f^{-1}V_i}$ is $f_i$-perfect for every $i \in \{1, \dots, n\}$, then it holds that $\CE \otimes^{\LL}_X \LL f^*\! \CM$ is bounded because all its restrictions to $f^{-1}V_i$ are. It follows that $\CE$ is $f$-perfect.

Assume now that $\CE$ is $f$-perfect. Denote by $j \colon V_i \inc Y$ the canonical embedding and let $\CN \in \Dbq{V_i}$. Let $r$ be the maximum integer such that $\h^r(\CN) \neq 0$. If $k > r$, the truncated\footnote{For the definition of $\tau^{\leq k}$, see \cite[\S 1.10]{yellow}.} complex $\CM := \tau^{\leq k}\R{}j_* \CN$ is bounded quasi-coherent and such that $\CM|_{V_i} \cong \CN$. As $\CE \otimes^{\LL}_X \LL f^* \CM$ is bounded, it keeps being so when restricted to $f^{-1}V_i$ and this implies that $\CE|_{f^{-1}V_i} \otimes^{\LL}_{V_i} \LL{}f_i^*\CN$ is bounded. Therefore we conclude that $\CE|_{f^{-1}V_i}$ is $f_i$-perfect.
\end{proof}
 
\begin{prop}\label{locchar}
 Let $f \colon X \to Y$ be a  morphism of schemes with $Y$ affine. Let $X = U_1 \cup \dots \cup U_n$ with $U_i$ an affine open subscheme  of $X$. The following conditions are equivalent for $\CE \in \Dbc{X}$:
\begin{enumerate}
 \item\label{locuno} The complex $\CE$ is $f$-perfect.
 \item\label{locdos} For every $i \in \{1, \dots, n\}$, the complex $\CE|_{U_i}$ is   $f|_{U_i}$-perfect.
 \item\label{loctres} For every $i \in \{1, \dots, n\}$, the complex   $\R{}f_*\CE_{U_i}$ is of finite flat dimension.
 \end{enumerate}
\end{prop}

\begin{proof} 
The implication \eqref{locuno} $\imp$ \eqref{locdos} is given by Proposition \ref{prop204}. The equivalence of \eqref{locdos} and \eqref{loctres} is given by Theorem \ref{absrelmor}, (\ref{f-perf1}). Finally, let us see that \eqref{locdos} $\imp$ \eqref{locuno}. For every $\CM \in \Dbq{Y}$, the complex $\CE \otimes^{\LL}_{X} \LL f^* \CM$ is bounded, since its restriction to each $U_i$ is bounded being $\CE|_{U_i}$ a $f|_{U_i}$-perfect complex.
\end{proof}

\begin{rem}
 For flat and locally of finite presentation morphism, the analogous to the previous two results for Lieblich's notion are given in \cite[\href{https://stacks.math.columbia.edu/tag/0DI2}{Tag 0DI2}]{sp} and \cite[\href{https://stacks.math.columbia.edu/tag/0GEH}{Tag 0GEH}]{sp}.
\end{rem}

\begin{cosa}\textbf{Transverse squares.}\label{transq}
Consider a Cartesian square of schemes
\begin{equation}\label{transverse}
\begin{tikzpicture}[baseline=(current  bounding  box.center)]
\matrix(m)[matrix of math nodes, row sep=2.6em, column sep=2.8em,
text height=1.5ex, text depth=0.25ex]{
  X' & X \\
  S' & S \\};
\path[->,font=\scriptsize,>=angle 90] 
(m-1-1) edge node[auto] {$g'$}
(m-1-2) edge node[left] {$f'$} (m-2-1)
(m-1-2) edge node[auto] {$f$} (m-2-2)
(m-2-1) edge node[auto] {$g$} (m-2-2);
\end{tikzpicture}
\end{equation}
We say that the square is \emph{transverse}, or sometimes \emph{tor}-independent (or just independent as in \cite[\S 3.10]{yellow}) if for all pair of points $s' \in S'$, $x \in X$ such that $s:= g(s') = f(x)$:
\[
\tor_i^{\CO_{S,s}}(\CO_{S',s'}, \CO_{X,x}) = 0, \quad \text{ for all } i > 0.
\]
For a transverse square, base change holds, namely, for $\CE \in \D_\qc(X)$, the base change morphism \cite[Proposition (3.7.2)]{yellow} is an isomorphism
\[
\theta \colon \R{}f'_*\LL{g'}^*\CE \liso \LL g^*\R{}f_*\CE .
\]
 For the proof see \cite[Theorem (3.10.3)]{yellow}.
\end{cosa}

\begin{prop}\label{pbpar}
 
Let
\begin{equation}
\begin{tikzpicture}[baseline=(current  bounding  box.center)]
\matrix(m)[matrix of math nodes, row sep=2.6em, column sep=2.8em,
text height=1.5ex, text depth=0.25ex]{
  X' & X \\
  S' & S \\};
\path[->,font=\scriptsize,>=angle 90] 
(m-1-1) edge node[auto] {$g'$}
(m-1-2) edge node[left] {$f'$} (m-2-1)
(m-1-2) edge node[auto] {$f$} (m-2-2)
(m-2-1) edge node[auto] {$g$} (m-2-2);
\end{tikzpicture}
\end{equation}
be a transverse square. 
If $\CE  \in \D_\qc(X)$ is $f$-perfect, then $\LL{g'}^*\CE$ is $f'$-perfect.
\end{prop}

\begin{proof} By Propositions  \ref{locchar0} and \ref{locchar}, we may assume that $S$, $S'$ and $X$ are affine, and then $f,g,f',g'$ are affine morphisms. By Theorem \ref{absrelmor}, (\ref{f-perf1}),  it is enough to check that $f'_*\LL{g'}^*\CE$  is of finite flat dimension. By  base change  $f'_*\LL{g'}^*\CE\cong \LL g^*f_*\CE$, and $f_*\CE$ is of finite flat dimension, because $\CE$ is $f$-perfect. Then $\LL g^*f_*\CE$ is also of finite flat dimension. Indeed, for any $\CM\in \D_\qc(S')$, $\LL g^*f_*\CE\otimes_{S'}^{\LL}\CM$ is bounded because it is bounded after applying $g_*$, by the projection formula.
\end{proof}

\begin{rem}
 For a flat and locally of finite presentation morphism $f$, the previous result for Lieblich's notion is given in \cite[\href{https://stacks.math.columbia.edu/tag/0DI5}{Tag 0DI5}]{sp}.
 \end{rem}

\begin{prop} \label{perffib}
 Let $X$ be a   scheme and $f \colon X \to S$ be a flat  
 morphism with regular fibres (\textit{e.g.} a smooth morphism). For $\CE \in \Dbc{X}$, the following are equivalent:
 
\begin{enumerate}
 \item \label{lisuno} The complex $\CE$ is perfect.
 \item \label{lisdos} The complex $\CE$ is $f$-perfect.
 \item \label{listres} For all $s \in S$, $\LL{}l_s^*\CE$ is bounded, with $l_s \colon X_s \to X$ the canonical map from the fiber $X_s := f^{-1}(\{s\})$.
\end{enumerate}
\end{prop}

\begin{proof}
The implication \eqref{lisuno} $\imp$ \eqref{lisdos} is a consequence of $f$ being flat. For \eqref{lisdos} $\imp$ \eqref{listres} it is enough to consider (with the notation of \ref{notres}) that $$l_{s *}\LL{}l_s^*\CE\cong  \CE \otimes^{\LL}_X l_{s *}\CO_{X_s}\cong \CE \otimes^{\LL}_X f^*\! \CK(s)  $$ where the second isomorphism is again due to the flatness of $f$.

Let us see \eqref{listres} $\imp$ \eqref{lisuno}. Let $x \in X$ and denote, as before, by $i_x \colon \spec(\kappa(x)) \to X$ the canonical morphism. By Theorem \ref{perfcar}, \eqref{cuatro} $\imp$ \eqref{uno}, it is enough to see that $\LL{}i_x^* \CE$ is bounded. Let $s = f(x)$ and  $l_x \colon \spec(\kappa(x)) \to X_s$ the canonical morphism. Notice that $i_x = l_s\circ l_x$. By assumption, $\LL{}l_s^*\CE$ is bounded. Since $X_s$ is regular, $\LL{}l_s^*\CE$ is perfect by Cororally \ref{cor105}. It follows that $\LL{}i_x^* \CE = \LL{}l_x^*\LL{}l_s^*\CE$ is bounded as wanted.
\end{proof}

\begin{rem}
 For a locally of finite presentation morphism $f$, the equivalence \eqref{lisdos} $\dimp$ \eqref{listres} for Lieblich's notion is given in \cite[\href{https://stacks.math.columbia.edu/tag/0GEH}{Tag 0GEH}]{sp}.
 \end{rem}

\section{Relatively perfection for quasi-proper morphisms}\label{sec6}

In this section we make a more exhaustive study of the behavior of relative perfect complexes with respect to quasi-proper maps.

\begin{prop} \label{consprop}
 Let $q \colon X \to \overline{X}$ be a quasi-proper morphism of $S$-schemes. Let $f \colon X \to S$ and $\overline{f} \colon \overline{X} \to S$ be the structure maps so that $\overline{f} \circ q = f$. Let $\CE \in \Dbc{X}$.
\begin{enumerate}[1.]
 \item If $\CE$ is $f$-perfect, then $\R{}q_*\CE$ is $\overline{f}$-perfect.
 \item Conversely, if $q$ is a finite morphism and $q_*\CE$ is $\overline{f}$-perfect, then $\CE$ is $f$-perfect. 
\end{enumerate}
\end{prop}

\begin{proof}
Let $\CM \in \Dbq{S}$. By the projection formula, it holds that
\[
\R{}q_*\CE \otimes^{\LL}_{\overline{X}}   \LL\overline{f}^* \CM \cong
\R{}q_*(\CE \otimes^{\LL}_{X} \LL{}f^*\CM)
\]
Since  $\CE$ is $f$-perfect, we conclude that $\R{}q_*\CE \otimes^{\LL}_{\overline{X}}\LL\overline{f}^* \CM  $ is bounded, i.e, $\R{}q_*\CE$ is $\overline{f}$-perfect. 

If $q$ is moreover finite, then it is also affine. In order to see that $\CE \otimes^{\LL}_{X} \LL f^* \CM$ is bounded, it is enough to check that $q_*(\CE \otimes^{\LL}_{X} \LL{}f^*\CM)$ is bounded. As before, this last complex is isomorphic to $q_*\CE \otimes^{\LL}_{\overline{X}} \LL\overline{f}^* \CM$ which is bounded by assumption.
\end{proof}

\begin{cor}\label{agreement}
Let $f \colon X \to S$ be a 
morphism of schemes such that there is a factorization $f = p \circ i$ where $i \colon X \inc  \overline{X}$ is a closed embedding such that $i_*\CO_X$ is pseudo-coherent on $\overline{X}$ and $p \colon \overline{X} \to S$ is a smooth   morphism. Then:
\begin{equation*}
 \CE \text{ is $f$-perfect } \ldimp i_*\CE \text{ is perfect }
\end{equation*}
\end{cor}

\begin{proof}
 Notice that $\CE$ is bounded pseudo-coherent if, and only if, $i_*\CE$ is bounded pseudo-coherent  (\cite[Lemme 1.1.1]{finrel}). By Proposition \ref{consprop}, $\CE$ is $f$-perfect if, and only if, $i_*\CE$ is $p$-perfect. We conclude by  Proposition \ref{perffib}.
\end{proof}

\begin{rem}
The previous Corollary, together with \cite[Proposition 4.4]{finrel}, shows that our definition of $f$-perfection agrees with the original one given in \cite[D\'efinition 4.1]{finrel}, when $f$ is a pseudo-coherent morphism, so there is no conflict of terminology. We have  used our characterization to begin with because it allows us to obtain the basic properties of relative perfect complexes in an easier way and in a very general setting.
\end{rem}

\begin{thm}\label{carrelperf}
Let $f \colon X \to Y$ be quasi-proper morphism and let $\CE \in \Dbc{X}$. The following are equivalent:
 
\begin{enumerate}
 \item \label{rpuno} $\CE$ is $f$-perfect.
 \item \label{rpdos} $\CE \otimes^{\LL}_{X} \LL{}f^*\!\CK(y)$ is bounded for all $y \in Y$.
 \item \label{rptres} $\R{}f_*(\CE \otimes^{\LL}_{X} \CL)$ is perfect for all $\CL \in \perf(X)$.
\end{enumerate}
If further $f$ is flat, then $\CE$ is $f$-perfect if, and only if, $\LL{}l_y^*\CE$ is bounded for all $y \in Y$ with $l_y \colon X_y \to X$ the canonical map from the fiber $X_y := f^{-1}(\{y\})$.
\end{thm}

\begin{proof}
 The implication \eqref{rpuno} $\imp$ \eqref{rpdos} is clear. Let us see \eqref{rpdos} $\imp$ \eqref{rptres}. If $\CL$ is perfect, then the complex $\CE \otimes^{\LL}_{X} \CL \otimes^{\LL}_{X} \LL{}f^*\!\CK(y)$ is bounded. Applying $\R{}f_*$ and the projection formula, we conclude that $\R{}f_*(\CE \otimes^{\LL}_{X} \CL) \otimes^{\LL}_{Y} \CK(y)$ is bounded too. As a consequence, $\R{}f_*(\CE \otimes^{\LL}_{X} \CL)$ is perfect by Theorem \ref{perfcar}.
 
Finally, \eqref{rptres} $\imp$ \eqref{rpuno}. Let $\CL \in \perf(X)$ and $\CM \in \Dbq{Y}$. We have the following isomorphisms
\begin{align*}
 \R{}f_*\R\!\shom_X^\bullet(\CL\!\!,\, \CE \otimes^{\LL}_X \LL{}f^* \CM) 
         &\cong \R{}f_*(\CL^{\vee}\!\!\otimes^{\LL}_X \CE \otimes^{\LL}_X \LL{}f^* \CM) \tag{\ref{perfdual}}\\
        &\cong  \R{}f_*(\CL^{\vee}\!\!\otimes^{\LL}_X \CE) \otimes^{\LL}_Y \CM \tag{projection} 
\end{align*}
and this last complex is bounded because $\R{}f_*(\CL^{\vee}\!\!\otimes^{\LL}_X \CE)$ is perfect by hypothesis. Now, $\R\Gamma(Y,-)$ is bounded, therefore the complex $\R\!\Hom_X^\bullet(\CL\!\!,\, \CE \otimes^{\LL}_X \LL{}f^* \CM)$ is bounded for all $\CL \in \perf(X)$ . By Proposition \ref{lem2011}, $\CE \otimes^{\LL}_X \LL{}f^* \CM$ is bounded, \ie\, $\CE$ is $f$-perfect.

If $f$ is flat, then $\LL{}f^*\CK(y) = l_{y *}\CO_{X_y}$ therefore $\CE \otimes^{\LL}_{X} \LL{}f^*\!\CK(y) = l_{y *}\LL{}l_y^*\CE$, which is bounded if, and only if, so $\LL{}l_y^*\CE$ is.
\end{proof}

\begin{rem} 
\begin{enumerate}[(a)]
 \item The previous Theorem may be seen as a relative version of Theorem \ref{perfcar} \eqref{uno} $\imp$ \eqref{tres} and as a sort of converse of Theorem \ref{absrelmor}.
 \item If $f\colon X\to Y$ is projective and $\CO_X(1)$ is a relatively very ample invertible sheaf on $X$, then condition (iii) can be weakened to the following:   $\R{}f_*(\CE \otimes_{X} \CO_X(r))$ is perfect for all $r\in \ZZ$. Indeed, the same proof \eqref{rptres} $\imp$ \eqref{rpuno} works by replacing Proposition \ref{lem2011} by the  Remark after Theorem \ref{biglem2011}.  See also  \cite[Proposition 1.6]{fmtgs}.
\end{enumerate}
\end{rem}

\begin{cor}
 Let $f \colon X \to Y$ be a quasi-proper
morphism of schemes. Then $f$ has finite flat dimension if and only if the functor $\R{}f_*$ takes perfect complexes to perfect complexes.  
\end{cor}

\begin{proof} Apply Theorem \ref{carrelperf}, \eqref{rpuno} $\dimp$ \eqref{rptres},  to $\CE=\CO_X$.
\end{proof}

\begin{rem}
 The previous result gives an alternate proof of the equivalence between (i) and (ii) of Theorem 1.2. in \cite{LN} for a quasi-proper map, since a quasi-perfect morphism $f$ (\cite[Definition 1.1]{LN}) is characterized as a morphism such that  $\R{}f_*$ takes perfect complexes to perfect complexes (\cite[Proposition 2.1]{LN}). Also, it gives the equivalence of our definition of finite flat dimension with the usual one of finite tor-dimension at least for proper maps. 
\end{rem}

\begin{cor}\label{properoverregular} Let $f \colon X \to Y$ be quasi-proper, with $Y$ a regular scheme. Then, any $\CE \in \Dbc{X}$ is $f$-perfect, hence,  \[ \Dbc{X}=\perf(f).\]
\end{cor}

\begin{proof} Assume $\CE$ is bounded and pseudo-coherent. For any $\CL \in \perf(X)$, the complex $\R{}f_*(\CE \otimes^{\LL}_{X} \CL)$ is bounded and pseudo-coherent, and then perfect by Corollary \ref{cor105}. Hence $\CE$ is $f$-perfect by Theorem \ref{carrelperf}, \eqref{rptres} $\imp$ \eqref{rpuno}. 
\end{proof}

\begin{prop}\label{comprodrel}
Let $f \colon X \to Y$ be a quasi-proper map of schemes  and let $\CE \in \Dbc{X}$. Then the complex $\CE$ is $f$-perfect if, and only if, the functor $\CE \otimes^{\LL}_{X} \LL{}f^*(-)$ commutes with products. 
\end{prop}

\begin{proof} Let $\{\CM_\lambda\}_{\lambda \in \Lambda}$ be a family of objects in $\D_\qc(X)$ and let 
\[ 
\CE \otimes^{\LL}_{X} \LL{}f^*(\prod_{\lambda\in\Lambda}\CM_\lambda) \overset\phi\to \prod_{\lambda\in\Lambda}  \CE \otimes^{\LL}_{X} \LL{}f^*(\CM_\lambda)
\] be the natural morphism. Since $\D_\qc(X)$ is generated by perfect complexes, $\phi$ is an isomorphism if and only if $\R{}f_*(\CL\otimes_X^{\LL} \phi)$ is an isomorphism for any $\CL \in \perf(X)$. By Proposition \ref{comprod}, the functor 
$\R{}f_*(\CL\otimes_X^{\LL}-) \colon \D_\qc(X) \to \D_\qc(Y)$ commutes with products. Therefore $\CE \otimes^{\LL}_{X} \LL{}f^*(-)$ commutes with products if and only if so does $\R{}f_*(\CL\otimes_X^{\LL}\CE \otimes^{\LL}_{X} \LL{}f^*(-))$ for any $\CL \in \perf(X)$. By the projection formula, this amounts to say that $\R{}f_*(\CL\otimes_X^{\LL}\CE)\otimes_Y^{\LL}- \colon \D_\qc(Y) \to \D_\qc(Y)$ commutes with products, \ie that $\R{}f_*(\CL\otimes_X^{\LL}\CE)$ is perfect, using Proposition \ref{comprod} again. We conclude by Theorem \ref{carrelperf}.
\end{proof}

\begin{rem}
 The previous Proposition is a relative version of Proposition \ref{comprod}. Notice that the products are computed within $\D_\qc(X)$ and $\D_\qc(Y)$, see \ref{prodqc}.
\end{rem}

\section{Semicontinuity and Grauert's theorem} \label{sec5}

In this section we show that semicontinuity and a certain form of base change of cohomology follow from Theorem \ref{absrelmor}(\ref{f-perf2}), provided we substitute the classical fiber of a complex with a derived version to be defined below.

\begin{cosa}\label{niderfib}
Let $f \colon X \to Y$ be a morphism of schemes. For any point $y \in Y$, let us denote $X_y := X \times_{Y} \spec(\kappa(y))$ and $l_y \colon X_y \to X$ the canonical morphism. 
Notice that $l_y$ is an affine morphism. 
For $\CE \in \D_\qc(X)$ denote 
\[
\CE(y) :=  \CE \otimes^{\LL}_X \LL{}f^*\CK(y)   . 
\]
Notice that $\CE(y)$ is supported on $X_y$ and $\h^q(\CE(y))$ is a quasi-coherent $l_{y *}\CO_{X_y}$-module.  There is a natural morphism (of comparison with the ususal derived fiber)
\begin{equation}\label{compfib1}
\CE(y) \lto l_{y *}\LL{}l_{y}^* \CE \cong \CE \otimes^{\LL}_X l_{y *}\CO_{X_y}.
\end{equation}
The complex $\CE(y)$ is the \emph{nice} derived fiber of $\CE$ over $y$. It has a better behavior with respect to cohomology than the classical derived fiber $\LL{}l_{y}^* \CE$.
 Notice that if $x \in f^{-1}(\{y\})$, then $\CE(y)_x = \CE_x \otimes^{\LL}_{\CO_{Y,y}} \kappa(y)$.
 
 By the projection formula $\R f_* \CE \otimes^{\LL}_Y \CK(y) \cong \R f_* (\CE  (y)) $ and then
\begin{equation}\label{cohfiber}
\LL i_y^* (\R f_* \CE) \cong \R{}\Gamma(X , \CE(y))
\end{equation}
 with $i_y \colon \spec(\kappa(y)) \to Y$ the canonical morphism. Hence $\h^q(X,\CE(y))$ are $\kappa(y)$-vector spaces.
 
 \begin{rem} The natural morphisms $\CE(y)\to l_{y *}l_y^{-1}\CE(y)\to \R l_{y *}l_y^{-1}\CE(y)$ are isomorphisms. This follows easily by factoring $l_y$ as the composition of the closed embedding $X_y\hookrightarrow X\times_Y\spec(\CO_{Y,y})$ and the flat morphism $X\times_Y\spec(\CO_{Y,y})\to  X$. Then
 \[  \R{}\Gamma(X , \CE(y)) \cong  \R{}\Gamma(X_y , \CE(y)\vert_{X_y}). \] This justifies the notation $H^q(X_y,\CE(y)):=H^q(X,\CE(y))$ that we shall use in the sequel.
 \end{rem}
\end{cosa}

\begin{ex}\label{ejemplo}
Let $Y$ be a regular scheme \emph{e.~g.}\/ a smooth algebraic scheme over a  field. Let $y_0 \in Y$ be a closed point. Let $f \colon X \to Y$ be the blow-up of $Y$ along $y_0$. It is customary to denote $E = X_{y_0}$ the \emph{exceptional fiber}. In this case the morphism $l_{y_0} \colon E \to X$ is just the canonical embedding. Let us consider $\CE = \CO_X(1)$. As a complex, $\CE$ is $f$-perfect, by Corollary \ref{properoverregular}. The usual fiber of $\CE$ at $y_0$ is $\CE_{y_0} : = l_{y_0}^* \CE= \LL l_{y_0}^* \CE =\CO_E (1)$ (notation as in \cite[III, \S 9]{HAG}). In contrast, the nice derived fiber is $\CE(y_0)= \CE\otimes_X^{\LL}\LL{}f^*\CK(y_0) $, whose $0^\mathrm{th}$ cohomology is $\CE_{y_0}$, but whose higher cohomologies take into account the lack of flatness of $f$ along $y$. For any $y\neq y_0$, one has $\CE(y)=\CE_y =\CK(y)$. An easy computation shows that the Euler characteristics are:
\[\aligned  
\chi(\CE_{y_0})& = \dim(Y), &\text{ while } \chi(\CE_{y})&=1 \text{ for any } y\neq y_0\\ 
\chi(\CE (y_0))& = 1, &\text{ and } \chi(\CE (y))&=1 \text{ for any } y\neq y_0.\endaligned\] 
which shows that $\chi(\CE_{y})$ is not constant along $Y$, but $\chi(\CE(y))$ is. Thus, $\CE$ does not satisfy the classical semicontinuity theorem; this is due to the fact that $\CE$ is not flat over $Y$.
\end{ex}

\begin{thm}\label{locconst}
 Let $f \colon X \to Y$ be a quasi-proper morphism of schemes. For $\CE \in \perf(f)$ the function
\begin{align*}
 \chi(f, \CE) \colon Y &\lto \ZZ\\
 y & \goesto \sum_{q \in \ZZ} (-1)^q \dim_{\kappa(y)} \h^q(X_y, \CE(y))
\end{align*}
is locally constant on $Y$.
\end{thm}

\begin{proof} By Theorem \ref{absrelmor},   $\R{}f_*\CE$ is a perfect complex; so $\R{}\Gamma(X_y , \CE(y))$   is a bounded complex and $H^q(X_y,\CE(y))$ are finite dimensional (as $\kappa(y)$-vector spaces) in view of \eqref{cohfiber}.  Thus the alternate sum in the definition of $\chi(f, \CE) $ makes sense. Let $y \in Y$ and let $U = \spec(A) \subset Y$ be an affine neighborhood of $y$ such that $\R{}f_*\CE|_U \cong \widetilde{P^\bullet}$ where $P^\bullet$ is a bounded complex of finite free $A$-modules. It is enough to check that the function $\chi(f, \CE)|_U$ is constant. Now,
\begin{align*}
 \chi(f, \CE)(y) &= \sum_{q \in \ZZ} (-1)^q \dim_{\kappa(y)} \h^q(P^\bullet \otimes_A \kappa(y)) \tag{by \eqref{cohfiber}}\\
 & = \sum_{q \in \ZZ} (-1)^q \dim_{\kappa(y)} (P^q \otimes_A \kappa(y))\\
 & = \sum_{q \in \ZZ} (-1)^q \rank(P^q)
\end{align*}
and this last sum does not depend on $y$.
\end{proof}

\begin{cosa}
On a scheme $X$, we say that a function $\chi \colon  X \to \ZZ$ is upper semicontinuous if all the sets  $U^{\chi}_n := \{x \in X \,/ \, \chi(x) < n\}$ are open, equivalently if the sets $Z^{\chi}_n := \{x \in X \,/ \, \chi(x) \geq n\}$ are closed. A function $\chi$ is lower semicontinuous if $-\chi$ is upper semicontinuous.
\end{cosa}

\begin{thm}\label{semicont}
 (Semicontinuity)  Let $f \colon X \to Y$ be a quasi-proper morphism of schemes. For $\CE \in \perf(f)$ the function
\begin{align*}
 h^p(f, \CE) \colon Y & \lto \ZZ\\
 y                    & \goesto \dim_{\kappa(y)} \h^p(X_y, \CE(y))
\end{align*}
is upper semicontinuous on $Y$ for all $p \in \ZZ$.
\end{thm}

\begin{proof}
 As before, $\R{}f_*\CE$ is perfect complex. For $y \in Y$ there is an  affine neighborhood $U = \spec(A) \subset Y$ such that $\R{}f_*\CE|_U \cong \widetilde{P^\bullet}$ where $P^\bullet$ is a bounded complex of finite free $A$-modules. Denote by $d$ the differential of the complex $P^\bullet$. There is a chain of equalities
\begin{align*}
 h^p(f, \CE)(y) & = \dim_{\kappa(y)} \h^p(X_y, \CE(y)) \\
     &= \dim_{\kappa(y)} \h^p(P^\bullet \otimes \kappa(y)) \tag{by \eqref{cohfiber}}\\
     &= \dim_{\kappa(y)} \ker(d^p \otimes \id) - \dim_{\kappa(y)} \Img(d^{p-1} \otimes \id)\\
     &= \dim_{\kappa(y)} (P^p \otimes \kappa(y)) - \dim_{\kappa(y)} \Img(d^{p} \otimes \id) - \dim_{\kappa(y)} \Img(d^{p-1} \otimes \id)
\end{align*}
that expresses $ h^p(f, \CE)$ as a  constant function subtracted by functions of the kind
\[
y \goesto \dim_{\kappa(y)} \Img(d^{p} \otimes \id).
\]
It is enough to check that these functions are lower semicontinuous. This follows because they are functions that assign the rank of linear maps between free modules and they  are lower semicontinuous (see, for instance, \cite[Chapter IV, \S 2, Corollary 2.6]{KCA}).
\end{proof}

The following lemma is well-known. We include a proof for the convenience of the reader.

\begin{lem}\label{projcar}
Let $Y$ be a  reduced scheme and let $\CF \in \SA_\qc(Y)$ be of locally finite presentation. Denote, as before, $\CF(y) := \CF_y \otimes_{\CO_{Y,y }} \kappa(y)
$. The function
\[
y \goesto \dim_{\kappa(y)} \CF(y) 
\] 
is locally constant on $Y$ if, and only if, $\CF$ is locally free.
\end{lem}

\begin{proof} The if part is immediate. For the converse, the question is local, so we reduce at once to the affine case $Y = \spec(A)$, where $A$ is a  reduced ring and $\CF = \widetilde{M}$ with $M$ a finitely presented module. Let $\ip \in \spec(A)$, let $\mu_\ip(M)$ denote the minimal number of generators of the $A_\ip$-module $M_\ip$. By Nakayama's Lemma
 \[
\mu_\ip(M) = \dim_{\kappa(\ip)} M \otimes_{A } \kappa(\ip)
 \]
Let $\{m_1, \ldots, m_n\} \subset M$ such that $\{m_1 \otimes 1, \ldots, m_n \otimes 1\}$ is a basis of $M \otimes_{A } \kappa(\ip)$, 
therefore $n=\mu_\ip(M)$. Let $\varphi \colon A^{n} \to M$ be the homomorphism given by $\varphi(e_i) := m_i$ with $\{e_1, \ldots, e_n\}$ denoting the canonical base of $A^n$. As the set $\{m_1, \ldots, m_n\}$ generate $M$ at a neighborhood of $\ip$, after localizing $A$ at an $f \notin \ip$ we can assume that $\varphi$ is surjective. To avoid clutter, we substitute $A$ and $M$ by its localized counterparts. If $\ker(\varphi) = 0$, we are done, as $M$ is free in a neighborhood of $\ip$ being isomorphic to $A^n$. Suppose not, then there is an element $(a_1, \ldots, a_n) \in \ker(\varphi)$ with some $a_i \neq 0$. But, as $A$ is reduced, this means that there is a prime ideal $\iq$ such that $a_i \notin \iq$, but this implies that $\frac{a_i}{1}$ is invertible in $A_{\iq}$ so $\mu_\iq(M) < n=\mu_\ip(M)$, a contradiction.
\end{proof}

\begin{cor}\label{grauert}
 (Grauert's theorem) In the setting of the previous Theorem, let us assume further that $Y$ is reduced. The function $h^p(f, \CE)$ is locally constant for some $p$ if, and only if, $\SR^p\!f_*\CE = \h^p(\R{}f_*\CE)$ is locally free and the canonical map
\begin{equation}\label{grauertiso}
(\SR^p\!f_*\CE)_y \otimes_{\CO_{Y,y}} \kappa(y) \lto \h^p(X_y, \CE(y))
\end{equation}
is an isomorphism for every $y \in Y$.
\end{cor}

\begin{proof} If $\SR^p\!f_*\CE$ is locally free and \eqref{grauertiso} is an isomorphism, then   $h^p(f, \CE)$ is locally constant by   Lemma \ref{projcar}. For the converse, we may assume that $Y$ is affine, the spectrum of a reduced ring $A$. We keep the notations of Theorem \ref{semicont}. Looking at the definition of $h^p(f, \CE)(y)$ in the displayed equalities of the Theorem, we conclude that the functions on $y$
 \[
 \dim_{\kappa(y)} \Img(d^{p} \otimes \id) \quad\text{ and }\quad
  \dim_{\kappa(y)} \Img(d^{p-1} \otimes \id)
 \] have to be locally constant. It follows that the function on $y$
 \[
 \dim_{\kappa(y)} \cok(d^{p} \otimes \id) =
 \dim_{\kappa(y)} (P^p \otimes \kappa(y)) - \dim_{\kappa(y)} \Img(d^{p} \otimes \id)
 \]
 is also locally constant. By right exactness of the functor $- \otimes_{A} \kappa(y) $, we have
 \[
 \cok(d^{p} \otimes \id) = \cok(d^{p}) \otimes \kappa(y)
 \]
 By Lemma \ref{projcar}, in view of \cite[Chapter IV, \S 2, Corollary 3.6]{KCA}, we conclude that $\cok(d^{p})$ is a projective module, and, as a consequence 
 \[P^{p+1} \cong \Img(d^{p}) \oplus \cok(d^{p}).\]
 As $P^{p+1}$ is projective by assumption, so is $\Img(d^{p})$ and it follows that 
 \[P^{p} \cong \ker(d^{p}) \oplus \Img(d^{p}).\]
Thus, $\ker(d^{p})$ is projective too. Notice that as, $\h^p(P) = \ker(d^{p})/\Img(d^{p-1})$ we have the following equality
\begin{align*}
 \dim_{\kappa(y)} &(\h^p(P^\bullet) \otimes_A \kappa(y)) =\\
          & = \dim_{\kappa(y)} (\ker(d^{p}) \otimes_A \kappa(y)) 
            - \dim_{\kappa(y)} (\Img(d^{p-1}) \otimes_A \kappa(y))
\end{align*}
Therefore, as $\SR^p\!f_*\CE = \widetilde{\h^p(P^\bullet)}$ applying the same result once more we conclude that $\SR^p\!f_*\CE$ is a locally free sheaf.

For the base change map, notice that the sequence
\[
P^{p-1} \xto{\,d^{p-1}} P^{p} \xto{\,d^{p}} P^{p+1} 
\]
splits as
\begin{small}
 \[
\ker(d^{p-1}) \oplus \Img(d^{p-1}) \lto
\Img(d^{p-1}) \oplus \h^p(P^\bullet) \oplus \Img(d^{p}) \lto
\Img(d^{p}) \oplus \cok(d^{p})
\]
\end{small}
from which it follows that
\[
\h^p(P^\bullet \otimes_A \kappa(y)) \cong \h^p(P^\bullet) \otimes_A \kappa(y)
\]
This isomorphism immediately yields
\[
\h^p(X_y, \CE(y)) \cong
(\SR^p\!f_*\CE)_y \otimes_{\CO_{Y,y}} \kappa(y)
\]
as wanted.
\end{proof}

\begin{rem}
\begin{enumerate}[(a)]
 \item Most of the references for the semicontinuity theorem rely on the observation that there is a perfect complex that computes cohomology, see, for instance \cite[Chapter III, Theorem 12.8]{HAG} or \cite[Chapter II, \S 5, Corollary, p.50]{MAV}. See  also \cite[\S 2.6]{leifu} whose strategy we followed for semicontinuity and Grauert's theorem. Our argument uses instead Theorem \ref{absrelmor}. With it, we see that the theorem is true without Noetherian hypothesis by using a suitable notion of derived fiber. We do not even make recourse to Noetherian approximation, apart from its implicit use via Kiehl's finiteness theorem.
 \item If $\CE$ is a pseudo-coherent $\CO_X$-Module flat over $Y$, Theorem \ref{semicont} yields the classical semicontinuity theorem because, in this case, for any $y \in Y$,
 \[
 \CE_y  := l_{y}^* \CE= \LL{}l_{y}^* \CE \quad\text{ and }\quad l_{y_*} \CE_y \cong \CE(y) \qquad \text{via \eqref{compfib1}.}
 \]
 \item If $\CE$ is a pseudo-coherent $\CO_X$-Module not necessarily flat over $Y$, then Theorem \ref{semicont} yields a semicontinuity theorem in terms of the nice derived fiber $\CE(y)$, for each $y \in Y$, as long as $\CE$ is $f$-perfect. This happens, for instance, when $Y$ is regular and $\CE$ is any pseudo-coherent sheaf on $X$ by Corollary \ref{properoverregular}. Notice also that
\[
 l_{y}^* \CE  = \h^0(\LL{}l_{y}^* \CE) \cong \h^0(\CE(y));
\]
 therefore, when $\CE$ is not flat over $Y$, the semicontinuity theorem still holds after replacing the ordinary fiber $l_{y}^* \CE$ by the nice derived fiber $\CE(y)$.
 \item If $f$ is flat, then $  l_{y_*}  \LL{}l_{y}^* \CE \cong \CE(y)$ and Theorem \ref{semicont} yields a semicontinuity theorem for the classic derived fibers $ \LL{}l_{y}^* \CE$, whenever $\CE$ is is $f$-perfect (again, if $Y$ is regular, it holds for any $\CE \in \Dbc{X}$).
 \item The same remarks apply to Theorems \ref{locconst} and \ref{grauert}.
\end{enumerate}
\end{rem}

\section{A bivariant theory}\label{bibthy}

We remind the reader our convention that all schemes are concentrated. Fix a base scheme $S$. The category of schemes over $S$ will be denoted by $\sch/S$. We define, for $f \colon X \to Y$ in $\sch/S$ the relative $K_0$-group as:
\[
\biE{f}{X}{Y} := K_0(\perf(f)).
\]
It will be convenient sometimes to abbreviate $\biE{f}{X}{Y} $ by $\BIK(f)$. Notice that, though a priori an element of $\BIK(f)$ is a (finite) formal sum of isomorphism classes of objects in $\perf(f)$, every element may be represented by a class of a single object.

The \emph{confined maps} of the theory are the quasi-proper morphisms. The \emph{independent squares} will be the   transverse   squares (\ref{transq}). With these choices, let us define the operations of the theory. 

\begin{enumerate}
\item[\textbf{Product}.] For any maps $f \colon X \to Y$, $g \colon Y \to Z$ in $\sch/S$, let us define the internal product 
\[
\BIK(f) \times \BIK(g) \overset{\cdot}{\lto} \BIK(g \circ f).
\]
Given $[\CE] \in \BIK(f)$ and $[\CF] \in \BIK(g)$ the product is defined by
\[
[\CE] \cdot [\CF] : = [\CE \otimes^{\LL}_{X} \LL{}f^*\CF].
\]
The complex $\CE \otimes^{\LL}_{X}\LL{}f^*\CF$ is $g \circ f$-perfect: if $\CM\in \Dbq{Z}$, then
\[ \CE \otimes^{\LL}_{X}\LL{}f^*\CF \otimes^{\LL}_{X}\LL{}(g\circ f)^*\CM \cong \CE \otimes^{\LL}_{X}\LL{}f^*(\CF \otimes^{\LL}_{Y}\LL{}g^*\CM)\] which is bounded because $\CF$ is $g$-perfect and $\CE$ is $f$-perfect.

For the case when $X=Y$ and $f = \id_X$,  $\CO_X \in \perf(X)$ is a left unit for the product we have just defined. Similarly when  $Y=Z$ and $g = \id_Y$, $\CO_Y \in \perf(Y)$ is a right unit.  

\item[\textbf{Pushforward}.]

Let $f \colon X \to Y$ and $g \colon Y \to Z$ be maps in $\sch/S$, with $f$~confined. 
The \emph{pushforward  by} {$f$}
\[
\pf{f} \colon \biE{g \circ f}{X}{Z} \lto \biE{g}{Y}{Z} 
\]
is defined for $[\CE] \in \BIK(g \circ f)$ by 
$\pf{f}[\CE] := [\R{}f_*\CE]$.
The definition is correct by virtue of Proposition \ref{pfper}.

\item[\textbf{Pullback}.]
Let $\dd$ be an independent square in $\sch/S$
  \[
   \begin{tikzpicture}[yscale=.95]
      \draw[white] (0cm,0.5cm) -- +(0: \linewidth)
      node (21) [black, pos = 0.41] {$Y'$}
      node (22) [black, pos = 0.59] {$Y$};
      \draw[white] (0cm,2.65cm) -- +(0: \linewidth)
      node (11) [black, pos = 0.41] {$X'$}
      node (12) [black, pos = 0.59] {$X$};
      \node (C) at (intersection of 11--22 and 12--21) [scale=0.9] {$\dd$};
      \draw [->] (11) -- (12) node[above, midway, sloped, scale=0.75]{$g'$};
      \draw [->] (21) -- (22) node[below, midway, sloped, scale=0.75]{$g$};
      \draw [->] (11) -- (21) node[left, midway, scale=0.75]{$f'$};
      \draw [->] (12) -- (22) node[right, midway, scale=0.75]{$f$};
   \end{tikzpicture}
  \]
The \emph{ pullback by $g,$ through $\dd,$}
\[
\pb{g}\colon \biE{f}{X}{Y} \lto \biE{f'}{X'}{Y'} 
\]
is given for $[\CE] \in \BIK(f)$ by 
$\pb{f}[\CE] := [\LL{}g'^*\CE]$.
The assignment is well-defined by virtue of Proposition \ref{pbpar}.
\end{enumerate}

\begin{thm}\label{biperfrel}
The previous data constitute a bivariant theory, with units, on $\sch/S$ taking values in abelian groups.
\end{thm}

We prove Theorem~\ref{biperfrel} by verifying that the axioms for a bivariant theory, as described in \cite[\S 2.2]{fmc}, do hold for the data referred to in that theorem. We will decompose the proof as a string of propositions \ref{A1}--\ref{A123}, each of them dealing with a corresponding axiom for the bivariant theory.

Following \cite{fmc}, we indicate an $\alpha\in\biE{f}{X}{Y}$ by the notation 
\vspace{-4pt}
\[
 \begin{tikzpicture}
      \draw[white] (0cm,1cm) -- +(0: \linewidth)
      node (22) [black, pos = 0.43] {$X$}
      node (23) [black, pos = 0.57] {$Y\>.$};
      \draw [->] (22) -- (23) node[auto, midway, scale=0.75]{$f$}
             node[below=1.5mm, midway, shape=circle, draw, scale=0.68]{$\!\alpha\!$};
 \end{tikzpicture}
\]

\begin{prop}\label{A1}
\emph{$(A_1)$ Associativity of  product:}\va3

For any sequence of maps in $\sch/S$ with corresponding elements in bivariant $K$-groups
  \[
   \begin{tikzpicture}
      \draw[white] (0cm,1cm) -- +(0: \linewidth)
      node (22) [black, pos = 0.29] {$X$}
      node (23) [black, pos = 0.43] {$Y$}
      node (24) [black, pos = 0.57] {$Z$}
      node (25) [black, pos = 0.71] {$W$};
      \draw [->] (22) -- (23) node[above=0.5mm, midway, scale=0.75]{$f$}
                              node[below=2mm, midway, shape=circle, draw,
                                   scale=0.68]{$\!\alpha\!$};
      \draw [->] (23) -- (24) node[above=0.5mm, midway, scale=0.75]{$g$}
                              node[below=2mm, midway, shape=circle, draw,
                                   scale=0.65]{$\!\!\!\!\beta\!\!\!\!$};
      \draw [->] (24) -- (25) node[above=0.5mm, midway, scale=0.75]{$h$}
                              node[below=2mm, midway, shape=circle, draw,
                                   scale=0.7]{$\!\!\gamma\!\!$};
   \end{tikzpicture}
  \]
we have, in\/  $\biE{h  g  f}{X}{W},$
\[
   (\alpha\< \cdot\<\< \beta)\<\< \cdot\<\gamma
   =  \alpha\< \cdot\<\< (\beta\< \cdot\< \gamma).
\]
\end{prop}

\begin{proof}
 Given $\CF \in \perf(f)$, $\CG \in \perf(g)$ and $\CH \in \perf(h)$, the desired relation reduces to the following isomorphisms:
\begin{align*}
 (\CF \otimes^{\LL}_{X} \LL{}f^*\CG)\<\< \otimes^{\LL}_{X}\<\< \LL{}(g f)^*\CH &\cong
 \CF \otimes^{\LL}_{X} (\LL{}f^*\CG \otimes^{\LL}_{X} \LL{}f^*\LL{}g^*\CH) \\
 & \cong \CF \otimes^{\LL}_{X} \LL{}f^*(\CG \otimes^{\LL}_{Y} \LL{}g^*\CH)
\end{align*}
 which is true by associativity and pseudo-functoriality of the derived inverse image functors \cite[(2.5.9) and (3.2.4)]{yellow}.
\end{proof}

\begin{prop}\label{A2}
\emph{($A_2$) Functoriality of pushforward:}\va2

For maps $f\colon X \to Y\<,$ $\>g \colon Y \to Z$ and\/ $h \colon Z \to W,$ in\, $\sch/S$ with $f$ and $g$ confined, and 
$\alpha \in \biE{h g f}{X}{W},$ 
one has, in  $\biE{h}{Z}{W},$
\[
\pf{(g \circ f)} (\alpha)=\pf{g\>}\pf{f} (\alpha).
\]
\end{prop}

\begin{proof}
Let $\CF \in \perf(hgf)$. The proof follows as a consequence  of the canonical isomorphism
\[
\R{}(g \circ f)_*\CF \cong \R{}g_*\R{}f_*\CF
\]
\ie{} pseudo-functoriality of derived direct images \cite[(3.6.4)$_*$]{yellow}.
\end{proof}

 \begin{prop} \label{A3}
 \emph{($A_3$) Functoriality of pullback:}\va2
 
For any diagram in $\sch/S$,  with independent squares, 
  \[
   \begin{tikzpicture}
      \draw[white] (0cm,2.5cm) -- +(0: \linewidth)
      node (12) [black, pos = 0.3] {$X''$}
      node (13) [black, pos = 0.5] {$X'$}
      node (14) [black, pos = 0.7] {$X$};
      \draw[white] (0cm,0.5cm) -- +(0: \linewidth)
      node (22) [black, pos = 0.3] {$Y''$}
      node (23) [black, pos = 0.5] {$Y'$}
      node (24) [black, pos = 0.7] {$Y$};
      \draw [->] (12) -- (13) node[above=1pt, midway, scale=0.75]{$h'$};
      \draw [->] (13) -- (14) node[above=1pt, midway, scale=0.75]{$g'$};
      \draw [->] (22) -- (23) node[below=1pt, midway, scale=0.75]{$h$};
      \draw [->] (23) -- (24) node[below=1pt, midway, scale=0.75]{$g$};
      \draw [->] (12) -- (22) node[left,  midway, scale=0.75]{$f''$};
      \draw [->] (13) -- (23) node[left,  midway, scale=0.75]{$f'$};
      \draw [->] (14) -- (24) node[right=.5pt,  midway, scale=0.75]{$f$}
                              node[left=4pt, midway, shape=circle, draw,
                                   scale=0.68]{$\!\alpha\!$};
   \end{tikzpicture}
  \]
one has, in  $ \biE{f''}{X''}{Y''},$
\[
   \pb{(g \circ h)}(\alpha)=\pb{h}\pb{g}(\alpha).
\]
\end{prop}

\begin{proof}
 Let us consider $\CF \in \perf(f)$. The assertion is a consequence of the isomorphism
\[
\LL{}(g' \circ h')^*\CF \cong \LL{}h'^*\LL{}g'^*\CF
\]
of pseudo-functoriality of derived inverse images \cite[(3.6.4)$^*$]{yellow}.

\end{proof}

\begin{prop}\label{A12}
\emph{($A_{12}$) Product and pushforward commute:}\va2

For any diagram in $\sch/S$  
\[
 \begin{tikzpicture}
    \draw[white] (0cm,0.5cm) -- +(0: \linewidth)
    node (21) [black, pos = 0.3 ] {$X$}
    node (22) [black, pos = 0.43] {$Y$}
    node (23) [black, pos = 0.56] {$Z$}
    node (24) [black, pos = 0.69] {$W$};
    \draw [->] (22) -- (23) node[above, midway, scale=0.75]{$g$};
    \draw [->] (23) -- (24) node[above, midway, scale=0.75]{$h$}
                            node[below=1mm, midway, shape=circle, draw,
                                                scale=0.65]{$\!\!\!\!\beta\!\!\!\!$};
    \draw [->] (21) -- (22) node[above, midway, scale=0.75]{$f$};
    \draw [->] (21) .. controls +(1,-1) and +(-1,-1) .. (23)
                              node[below=1mm, midway, shape=circle, draw,
                                                     scale=0.68]{$\!\alpha\!$}
                              node[above, midway, scale=0.75]{$g \circ f$};
 \end{tikzpicture}
\]
with $f\colon X \to Y$ confined, one has, in $\biE{h g}{Y}{W},$
\[
\pf{f} (\alpha\<\cdot\<\< \beta )=\pf{f} (\alpha)\<\cdot\<\< \beta
\]
\end{prop}

\begin{proof}
 It is enough to consider $\CF \in \perf(gf)$ and $\CH \in \perf(h)$. The assertion is true by the following chain of isomorphims
 \begin{align*}
 \R{}f_*(\CF \otimes^{\LL}_{X} \LL{}(g \circ f)^*\CH) &\cong
 \R{}f_*(\CF \otimes^{\LL}_{X} \LL{}f^*\LL{}g^*\CH) \\
 & \cong \R{}f_*\CF \otimes^{\LL}_{Y} \LL{}g^*\CH
\end{align*}
in which we use pseudo-functoriality of derived inverse image and the projection formula.
\end{proof}

\begin{prop}
\label{A13}
\emph{($A_{13}$) Product and pullback commute:}\va3

For any diagram in $\sch/S$,  with independent squares, 
\begin{equation*} 
 \begin{tikzpicture}scale=1.3
      \draw[white] (0cm,0.5cm) -- +(0: \linewidth)
      node (E) [black, pos = 0.4] {$Z'$}
      node (F) [black, pos = 0.59] {$Z$};
      \draw[white] (0cm,2.65cm) -- +(0: \linewidth)
      node (G) [black, pos = 0.4] {$Y'$}
      node (H) [black, pos = 0.59] {$Y$};
      \draw[white] (0cm,4.8cm) -- +(0: \linewidth)
      node (I) [black, pos = 0.4] {$X'$}
      node (J) [black, pos = 0.59] {$X$};
      \draw [->] (G) -- (H) node[above, midway, sloped, scale=0.75]{$h'$};
      \draw [->] (E) -- (F) node[below=1pt, midway, sloped, scale=0.75]{$h$};
      \draw [->] (G) -- (E) node[left, midway, scale=0.75]{$g'$};
      \draw [->] (H) -- (F) node[right=1pt, midway, scale=0.75]{$g$}
                            node[left=2pt, midway, shape=circle, draw,
                                   scale=0.68]{$\!\beta\!$};
      \draw [->] (I) -- (J) node[above, midway, sloped, scale=0.75]{$h''$};
      \draw [->] (I) -- (G) node[left, midway, scale=0.75]{$f'$};
      \draw [->] (J) -- (H) node[right=1pt, midway, scale=0.75]{$f$}
                            node[left=4pt, midway, shape=circle, draw,
                                   scale=0.68]{$\!\alpha\!$};

 \end{tikzpicture}
\end{equation*}
one has, in $ \biE{g'\<\<f'}{X'}{Z'},$
\[
   \pb{h}(\alpha\<\cdot\<\< \beta)=\pb{h'}(\alpha)\<\cdot\<\pb{h}(\beta).
\]
\end{prop}

\begin{proof}
Let $\CF \in \perf(f)$ and $\CG \in \perf(g)$. The assertion is true by the following chain of isomorphisms
\begin{align*}
 \LL{}h''^*(\CF \otimes^{\LL}_{X} \LL{}f^*\CG) &\cong
          \LL{}h''^*\CF \otimes^{\LL}_{X} \LL{}h''^* \LL{}f^*\CG) \\
  & \cong \LL{}h''^*\CF \otimes^{\LL}_{X} \LL{}f'^* \LL{}h'^*\CG)
\end{align*}
where we use  pseudo-functoriality and monoidal property of derived inverse image.
\end{proof}

\begin{prop} \emph{($A_{23}$) Pushforward and pullback commute:}\va2

For any diagram in $\sch/S$,  with independent squares and with\/ $f$ confined, 
\begin{equation*} 
 \begin{tikzpicture}
      \draw[white] (0cm,0.5cm) -- +(0: \linewidth)
      node (E) [black, pos = 0.4 ] {$Z'$}
      node (F) [black, pos = 0.59] {$Z$};
      \draw[white] (0cm,2.65cm) -- +(0: \linewidth)
      node (G) [black, pos = 0.4 ] {$Y'$}
      node (H) [black, pos = 0.59] {$Y$};
      \draw[white] (0cm,4.8cm) -- +(0: \linewidth)
      node (I) [black, pos = 0.4 ] {$X'$}
      node (J) [black, pos = 0.59] {$X$};
      \node (label a) at (intersection of I--H and G--J) [scale=0.75] {$\da$};
      \node (label b) at (intersection of G--F and H--E) [scale=0.75] {$\db$};
      \draw [->] (G) -- (H) node[above, midway, scale=0.75]{$h'$};
      \draw [->] (E) -- (F) node[below, midway, scale=0.75]{$h$};
      \draw [->] (G) -- (E) node[left,  midway, scale=0.75]{$g'$};
      \draw [->] (H) -- (F) node[right,  midway, scale=0.75]{$g$};
      \draw [->] (I) -- (J) node[above, midway, scale=0.75]{$h''$};
      \draw [->] (I) -- (G) node[left,  midway, scale=0.75]{$f'$};
      \draw [->] (J) -- (H) node[right,  midway, scale=0.75]{$f$};
      \draw [->] (J) .. controls +(1.25,-1.25) and +(1.25,1.25) .. (F)
                       node[right=4pt, midway, shape=circle, draw,
                                   scale=0.68]{$\!\alpha\!$}
                       node[left=.5pt,  midway, scale=0.75]{$g  f$};                           
 \end{tikzpicture}
\end{equation*}
one has, in\/ $ \biE{g'}{Y'}{Z'},$
\[
\pf{f'\!\!}(\pb{h}\<(\alpha))=\pb{h}(\pf{f}\<(\alpha)). 
\]
\end{prop} 

\begin{proof}
 Let $\CF \in \perf(g \circ f)$. The formula follows from the isomorphism 
\[
 \R{}f'_*\LL{}h''^*\CF  \cong \LL{}h'^*\R{}f_*\CF
\]
\ie\/ base change isomorphism for independent squares \cite[(3.10.3)]{yellow}.
\end{proof}

\begin{prop}\label{A123} 
\emph{($A_{123}$) Projection formula:}\va2

For any diagram in $\sch/S$, with $\dd$ an independent square and\/ $g$ confined,
  \[\mkern-75mu
   \begin{tikzpicture}[xscale=1.1, yscale=1.05]
      \draw[white] (0cm,0.5cm) -- +(0: .75\linewidth)
      node (22) [black, pos = 0.35] {$Y'$}
      node (23) [black, pos = 0.6] {$Y$}
      node (24) [black, pos = 0.85] {$Z$};
      \draw[white] (0cm,2.8cm) -- +(0: .75\linewidth)
      node (12) [black, pos = 0.35] {$X'$}
      node (13) [black, pos = 0.6] {$X$};
      \draw [->] (12) -- (13) node[above, midway, scale=0.75]{$g'$};
      \draw [->] (22) -- (23) node[below, midway, scale=0.75]{$g$};
      \draw [->] (23) -- (24) node[below, midway, scale=0.75]{$h$};
      \draw [->] (12) -- (22) node[left,  midway, scale=0.75]{$f'$};
      \draw [->] (13) -- (23) node[right=1pt,  midway, scale=0.75]{$f$}
                              node[left=4pt, midway, shape=circle, draw,
                                   scale=0.68]{$\!\alpha\!$};
      \draw [->] (22) .. controls +(1,-1.1) and +(-1,-1.1) .. (24)
                              node[below=5pt, midway, shape=circle, draw,
                                   scale=0.68]{$\!\beta\!$}
                              node[above, midway, scale=0.75]{$h g$};
      \node (A) at (intersection of 22--13 and 23--12) [scale=0.75] {$\dd$};                                                  
   \end{tikzpicture}
  \]
one has, in\/ $\biE{h \circ f}{X}{Z}.$
 \[
  g'_\star(\pb{g}\alpha\<\cdot\<\< \beta)=\alpha\<\cdot\pf{g\>}(\beta).
 \]
\end{prop}

\begin{proof}
 It is enough to consider $\CF \in \perf(f)$ and $\CG \in \perf(h g)$.
It is a consequence of the following chain of isomorphisms
\begin{align*}
 \R{}g'_*(\LL{}g'^*\CF \otimes^{\LL}_{X'} \LL{}f'^*\CG) &\cong
            \CF \otimes^{\LL}_{X'} \R{}g'_*\LL{}f'^*\CG)  \\
    & \cong \CF \otimes^{\LL}_{X} \LL{}f^*\R{}g_*\CG
\end{align*}
where the first isomorphism is the projection formula (obviously!) and the second is applying again base change for independent squares \cite[(3.10.3)]{yellow}.
\end{proof}

\begin{cosa} \textbf{Orientation}.
For every $f \colon X \to Y$ in $\sch/S$ of finite flat dimension, define
 \[
 o(f) : = [\CO_X] \in \BIK(f).
 \]
 This makes sense because in this case $\CO_X$ is $f$-perfect. An important feature of a bivariant theory is the existence of an orientation for a certain class of orientable maps \cite[Definition 2.6.2]{fmc}. We have the following
\end{cosa}

\begin{prop}
 The class of finite flat dimension morphisms are orientable for the bivariant theory $\BIK$. For such an $f$, the element $o(f) \in \BIK(f)$ defines an orientation .
\end{prop}

\begin{proof}
 We have to check two things. First, let $f \colon X \to Y$ and $g \colon Y \to Z$ be morphisms of finite flat dimension. We have
 \[
o(f) \<\cdot\<\< o(g) = [\CO_X] \<\cdot\<\< [\CO_Y] = [\CO_X \otimes^{\LL}_{X} \LL{}f^*\CO_Y] = [\CO_X] = o(g \circ f).
 \]
 Secondly, $o(\id_X) = [\CO_X]$ and this is the identity element of the ring 
 \[
 \biE{\id_X}{X}{X}
 \]
 as wanted.
\end{proof}


\end{document}